\documentclass[12pt]{article}
\usepackage{authblk}
\usepackage{bbm}

\usepackage{wasysym}

\usepackage{url}
\usepackage{hyperref} 
\hypersetup{
    colorlinks=true,
    linkcolor=red,
    filecolor=magenta,      
    urlcolor=cyan,
    citecolor=red,
    }

\usepackage{framed}
\usepackage[normalem]{ulem}

\usepackage[utf8]{inputenc}
\usepackage{scalerel}
\usepackage{enumerate}
\usepackage{authblk}
\usepackage{wrapfig}
\usepackage{tikz-cd}
\usepackage[new]{old-arrows}
\usepackage{amsmath,amscd}
\usepackage{amsthm}
\usepackage{amsfonts}
\usepackage{color}
\usepackage{thmtools}
\usepackage{thm-restate}
\usepackage{amssymb}
\usepackage[margin=1in]{geometry}
\usepackage{enumitem}
\usepackage{cleveref}

\newtheorem{theorem}{Theorem}
\newtheorem{proposition}{Proposition}[section]
\newtheorem{corollary}{Corollary}[section]
\newtheorem{lemma}{Lemma}[section]
\newtheorem{remark}{Remark}[section]
\newtheorem{claim}{Claim}[section]
\newtheorem{question}{Question}
\newtheorem{conjecture}{Conjecture}

\theoremstyle{definition}
\newtheorem{definition}{Definition}
\newtheorem{example}{Example}[section]

\newcommand{\R}{\mathbb{R}}
\newcommand{\N}{\mathbb{N}}

\newcommand{\Sp}{\mathbb{S}}
\newcommand{\diam}{\mathrm{diam}}
\newcommand{\dis}{\mathrm{dis}}

\newcommand{\dgh}{d_\mathrm{GH}}
\newcommand{\dha}{d_\mathrm{H}}
\newcommand{\conv}{\mathrm{Conv}}

\newcommand{\convsph}{\mathrm{ConvSph}}

\newcommand{\disc}{\mathrm{disc}}

\DeclareMathOperator{\arccot}{arccot}
\DeclareMathOperator{\Ima}{Im}

\definecolor{darkblue}{rgb}{0.0, 0.0, 0.8}
\definecolor{darkred}{rgb}{0.8, 0.0, 0.0}
\definecolor{darkgreen}{rgb}{0.0, 0.8, 0.0}

\usepackage{fontawesome5}

\newcommand{\extension}{$\text{\faClock}$}

\title{Embedding-Projection Correspondences for the estimation of the Gromov-Hausdorff distance}
\author{Facundo M\'emoli and Zane Smith}
\date{Working paper\\\texttt{Version 1 \\ \today}}

\begin{document}

\maketitle

\begin{abstract}
This writeup describes ongoing work on designing and testing a certain family of correspondences between compact metric spaces that we call \emph{embedding-projection correspondences} (EPCs). Of particular interest are EPCs between spheres of different dimension.
\end{abstract}

\tableofcontents

\section{\label{sec:intro}Introduction}

\begin{framed}
\noindent This is work is evolving. It might contain  an incomplete account in some places. We will be updating this document frequently. The tag \extension \, 
 indicates that an upcoming update is planned in the corresponding part of the text. To indicate that accompanying code exists to demonstrate a construction/idea we will use the symbol \faLaptopCode.
\end{framed}

The central question we explore is: 
\begin{question}[\cite{lim2021gromov}]\label{q:central}
What is the precise value of the Gromov-Hausdorff distance $\dgh(\Sp^m,\Sp^n)$ between spheres of different dimensions (endowed with their geodesic distance)? 
\end{question}

Several results have been obtained which provide partial answers to \Cref{q:central}:
\begin{itemize}
    \item In \cite{lim2021gromov} Lim, M\'emoli and Smith obtain the precise value of $\dgh(\Sp^1,\Sp^2)$, $\dgh(\Sp^2,\Sp^3)$ and $\dgh(\Sp^1,\Sp^3)$. They also provide a lower bound for $\dgh(\Sp^m,\Sp^n)$ for all spheres which is based on a version of the Borsuk-Ulam theorem due to Dubins and Schwarz \cite{dubins1981equidiscontinuity}.  These lower bounds arise as obstructions for odd functions from $\Sp^n\to \Sp^m$ to be continuous, when $n>m$.

\item The currently best known lower bound for $\dgh(\Sp^m,\Sp^n)$ was found in \cite{adams2022gromov} via ideas which combine insights from \cite{lim2021gromov} and \cite{adams2020metric} in a way that that generalizes the version of the Borsuk-Ulam theorem due to Dubins and Schwarz.

    \item  In \cite{harrison2023quantitative} Jeffs and Harrison obtain the precise value of $g_{m,n}:=\dgh(\Sp^1,\Sp^{2k})$ for all  integers $k\geq 1$. 

    \item In an upcoming update to \cite{harrison2023quantitative}, Jeffs and Harrison also obtain the precise value of $\dgh(\Sp^1,\Sp^{2k+1})$ for all  integers $k\geq 1$.

\end{itemize}
\medskip

This writeup describes and develops several yet unpublished ideas that led to some of the results in our paper \cite{lim2021gromov}.  A historical account is given in \Cref{sec:hist}. \Cref{fig:dm} describes the current knowledge about the different values of $g_{m,n}.$

\begin{figure}
    \centering
    \includegraphics[width = 0.7\textwidth]{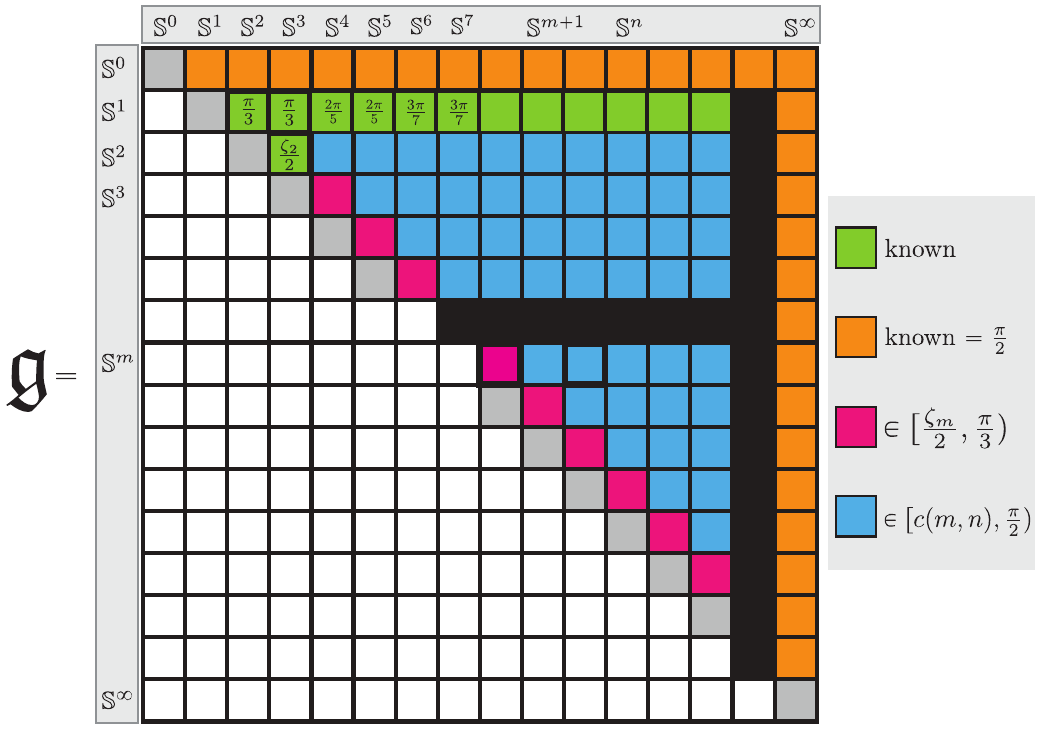}
    \caption{Current knowledge (as of \today) about the  value of $g_{m,n}$ for $0\leq m,n\leq \infty$. See \Cref{sec:intro}.}
    \label{fig:dm}
\end{figure}

\subsection{Acknowledgements}
We thank Henry Adams for suggesting to us exploring the connection between our ideas and the TMC and the Barvinok-Novik polytope.  That fruitful suggestion led to the general description of $R_{\gamma_n}$ that we are exploring in this project and to some of the results that we described. We also thank Sunhyuk Lim for careful reading of the first version and for his feedback.

This work is supported by  NSF-2310412, NSF-1547357, NSF-1740761, NSF-1901360 and by BSF 2020124.


\section{Background}

Given two sets $X$ and $Y$, a correspondence between them is any relation $R\subseteq X\times Y$ such that $\pi_X(R)=X$ and $\pi_Y(R)=Y$ where $\pi_X:X\times Y\rightarrow X$ and $\pi_Y:X\times Y\rightarrow Y$ are the canonical projections. Given two bounded metric spaces $(X,d_X)$ and $(Y,d_Y)$, and any non-empty relation $R\subseteq X\times Y$, its \emph{distortion} is defined as $$\dis(R):=\sup_{(x,y),(x',y')\in R}\big|d_X(x,x')-d_Y(y,y')\big|.$$
\begin{remark}\label{rem:mono-dis}
    Note that for any two nested non-empty relations $S\subset R$
between $X$ and $Y$ one has $\dis(S)\leq \dis(R).$
\end{remark}
For a function $\varphi:X\to Y$, its distortion, $\dis(\varphi)$ is just the distortion of its graph: $\dis(\varphi):=\dis(\mathrm{graph}(\varphi)).$ 

\begin{definition}
The Gromov-Hausdorff distance  between any two bounded metric spaces $(X,d_Y)$ and $(Y,d_Y)$ is defined as
\begin{equation}\label{eqn:dghaltdef}
    \dgh(X,Y):=\frac{1}{2}\inf_R \dis(R),
\end{equation}
where $R$ ranges over all correspondences between $X$ and $Y$.
\end{definition}

Given a correspondence $R$ between $X$ and $Y$, we say that a function $\psi:Y\to X$ is \emph{subordinate} to $R$ whenever its graph is contained in the correspondence:
$$\mathrm{graph}(\psi):=\{(\psi(y),y)|y\in Y\}\subset R.$$

In general, given a function $\psi:Y\to X$ its graph may fail to yield a correspondence between $X$ and $Y$. However, this of course is guaranteed whenever $\psi$ is surjective.

\begin{definition}[Modulus of discontinuity, \cite{dubins1981equidiscontinuity}]\label{def:mod-disc}
Let $Y$ be a topological space, $X$ be a metric space, and $f:Y\rightarrow X$ be any function. Then, we define $\mathrm{disc}(f)$, the \emph{modulus of discontinuity of $f$} as follows:
$$\mathrm{disc}(f):=\inf\{\delta\geq 0:\forall\,y\in Y,\exists\,\text{an open neighborhood }U_y\text{ of }y\text{ s.t. }\diam(f(U_y))\leq\delta\}.$$
Here, for a non-empty subset $A$ of $X$, its \emph{diameter} is defined as $\diam(A):=\sup_{a,a'\in A}d_X(a,a').$
\end{definition}

\begin{remark}
Of course, $\mathrm{disc}(f)=0$ if and only if $f$ is continuous.
\end{remark}

The modulus of discontinuity of a function is controlled by its distortion.
\begin{proposition}[{\cite[Proposition 5.3]{lim2021gromov}}]\label{prop:moddis}
Let $R$ be any correspondence between the metric spaces $X$ and $Y$ and let $\psi:Y\to X$ be any subordinate function. Then,
$$\mathrm{disc}(\psi)\leq \dis(\psi)\leq\mathrm{dis}(R).$$
\end{proposition}

\medskip
From now on, by $\mathcal{M}$ we will denote the collection of all compact metric spaces.

\begin{theorem}[{\cite[Main Theorem and  Theorem 5.1]{adams2022gromov}}] \label{thm:lower-bound} For all integers $k\geq 1$ and for any correspondence $R$ between $\Sp^1$ and $\Sp^{2k+1}$ one has $\dis(R)\geq \delta_{k}$ where $$\delta_k:= \frac{2\pi k}{2k+1}.$$
Similarly, 
$\dis(R)\geq \delta_k$ for any correspodence between $\Sp^1$ and $\Sp^{2k}.$
\end{theorem}

\begin{remark}
    Note that:
\begin{itemize}
\item the parameter $\delta_k$ above coincides with the edge length (in the geodesic sense) of an odd regular polygon inscribed in the unit circle $\Sp^1$. For example, $\delta_1=\tfrac{2\pi}{3}$, which corresponds to the distance between vertices of an equilateral triangle inscribed in the unit circle. 
\item the theorem of course implies that $\dgh(\Sp^1,\Sp^{2k})$ and $\dgh(\Sp^1,\Sp^{2k+1})$ are both bounded below by $\tfrac{\delta_k}{2}.$
\end{itemize}
\end{remark}

\Cref{thm:lower-bound} above is intimately related to the following complementary theorem. Recall that, for integers $n\geq m$, a function $f:\Sp^n\to \Sp^m$ is said to be antipode preserving (or just antipodal), if $f(-x) = - f(x)$ for all $x\in \Sp^n$.
\begin{theorem}[{\cite[Theorems 1.3 and 5.1]{adams2022gromov}}]\label{thm:disc-lower-bound}
Let $f:\Sp^{2k+1}\to \Sp^1$ be any antipodal function. Then $\mathrm{disc}(f)\geq \delta_k.$ Similarly, $\mathrm{disc}(f)\geq \delta_k$ for any antipodal function $f:\Sp^{2k}\to \Sp^1.$
\end{theorem}

The relationship between \Cref{thm:lower-bound} and \Cref{thm:disc-lower-bound} is explained in \cite[Remark 7.3]{adams2022gromov}; see also \cite[Section 5]{lim2021gromov}. 
\section{The general construction of EPCs}

Let $X,Y\in\mathcal{M}$ be two compact metric spaces such that there exists an embedding $\iota:X\to Y$  of $X$ into $Y$ (i.e. a homeomorphism onto its image). Importantly,  \emph{one does not require the embedding to be isometric}.  Consider then any \emph{closest point projection function} $\Pi_\iota:Y\to \iota(X)$: that is, $\Pi_\iota$ satisfies that for any $y\in Y$, $\Pi_\iota(y)$ is such that $$\rho_{\Pi_\iota}(y) := d_Y(y,\Pi_\iota(y)) = \min_{y'\in \iota(X)} d_Y(y,y').$$  Since $\iota(X)\subset Y$, any such function $\Pi_\iota$ is always surjective but it might not be injective. 

\begin{definition}\label{def:epc} The \emph{embedding-projection correspondence} induced by $\iota$ is the correspondence $R_\iota \in \mathcal{R}(X,Y)$ defined as:
$$R_\iota:=\{(x,y)\in X\times Y|\, \iota(x)=\Pi_\iota(y)\}.$$
By $\psi_\iota:Y\twoheadrightarrow X$ we will denote the function $Y\ni y \mapsto \psi_\iota(y):=\iota^{-1}(\Pi_\iota(y))$.   See \Cref{fig:epc}.
\end{definition}

\begin{figure}
    \centering
    \includegraphics[width=\linewidth]{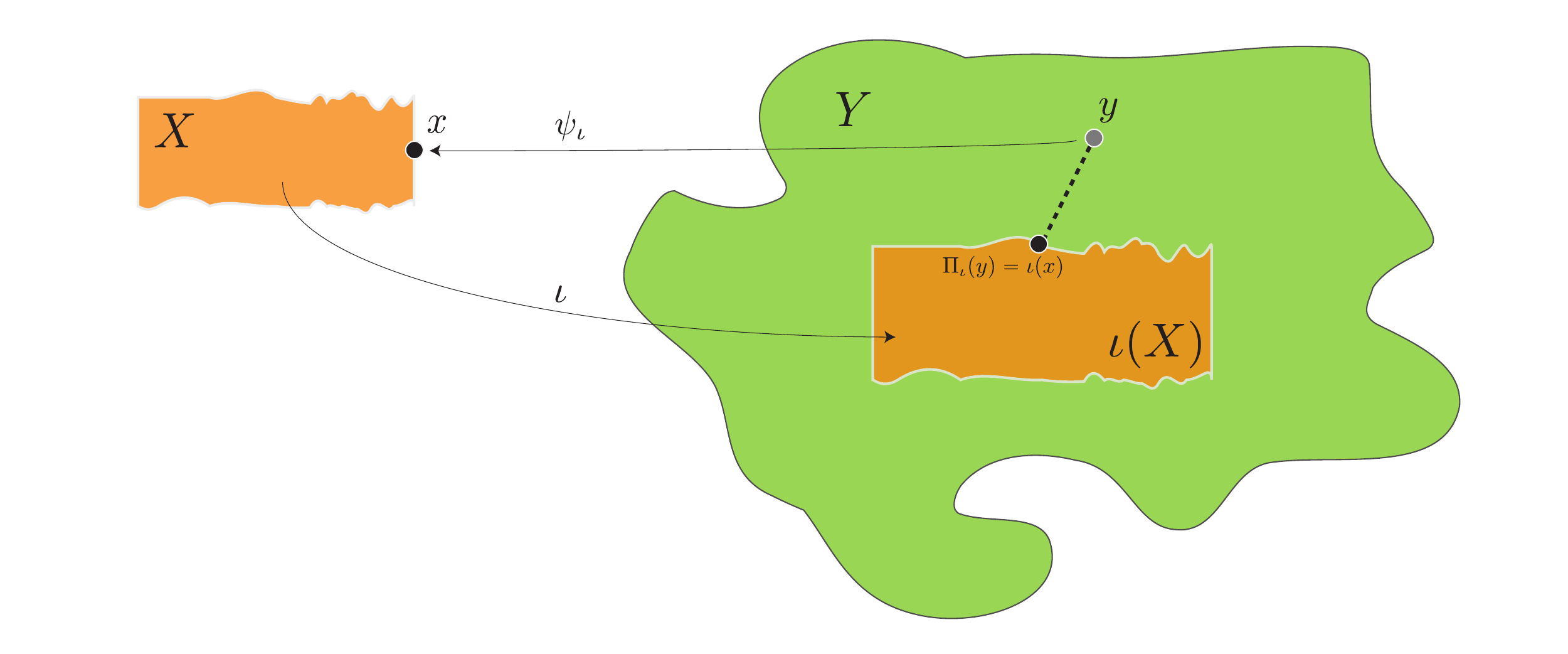}
    \caption{An embedding-projection correspondence. See \Cref{def:epc}.}
    \label{fig:epc}
\end{figure}

We will use the acronym EPC to denote correspondences of the  form described in the definition above and, similarly, we will use the acronym EPS to denote the associated surjections.
\begin{remark}\label{rem:map-corresp} Note that the function $\psi_\iota$ constructed above is subordinate to $R_\iota$ and that, furthermore,  $R_\iota = \mathrm{graph}(\psi_\iota)$ and therefore $\dis(R_\iota)=\dis(\psi_\iota).$ 
\end{remark}

The main motivation behind this definition is to carefully design the embedding $\iota$ so that the distortion of $R_\iota$ is as small as possible. It is not necessarily the case that an isometric embedding $\iota:X\hookrightarrow Y$ will give rise to a low distortion correspondence. Consider for example the case of any equatorial embedding $\iota:\Sp^1\hookrightarrow \Sp^2$. In that case, it is immediate to check that $\dis(R_\iota) = \pi$ which is far from the minimal possible distortion which is known to be $\frac{2\pi}{3}$ \cite[Proposition 1.16]{lim2021gromov}. In fact, results in the the same reference prove that this naive equatorial embedding fails to give good upper bounds in general.

\subsection{Interpretation of EPCs}
\label{sec:interp-EPC}
    For an EPC $R_\iota$ to be ``good" in the sense of having small distortion, it is necessary that both $\iota:X\to Y$ does not distort distances too much and that $\iota(X)$ provides an efficient covering of $Y$. 
    
    \begin{proposition}
        For any $EPC$ $R_\iota$, one has  
    $$\dis(R_\iota)\geq \max\big(\dis(\iota),\rho(\Pi_\iota)\big),$$
where $\rho(\Pi_\iota):=\sup_{y\in Y}\rho_{\Pi_\iota}(y)$ is the \emph{covering radius} of the projection function $\Pi_\iota$.
    \end{proposition}
\begin{proof}   
That $\dis(R_\iota)\geq \dis(\iota)$ is clear since $\psi_\iota|_{\iota(X)} = \mathrm{id}_X$ and therefore
$$\dis(R_\iota) = \dis(\psi_\iota)\geq \dis(\psi_\iota|_{\iota(X)}) = \sup_{x,x'\in X}\big|d_X(x,x')-d_Y(\iota(x),\iota(x'))\big|=\dis(\iota).$$  
To obtain $\dis(R_\iota)\geq \rho(\Pi_\iota)$ notice that for any $y\in Y$, $\psi_\iota(y) = \psi_\iota\big(\Pi_\iota(y)\big)$ so that 
$$\rho(\Pi_\iota) = \sup_{y\in Y}\big|d_X(\psi_\iota(y),  \psi_\iota(\Pi_\iota(y)))-d_Y(y,\Pi_\iota(y))\big|\leq \dis(R).$$
\end{proof}

\subsection{Voronoi cells and modulus of discontinuity of $\psi_\iota$}

For each  $x\in X$ consider the $x$-fiber of $\psi_\iota$: $$V_x:=\{y\in Y|\,\Pi_\iota(y)=\iota(x)\} = \{y\in Y|\,\psi_\iota(y)=x\}.$$ In other words, $\overline{V_x}$ is the  Voronoi cell induced by $\iota(x) \in \iota(X)$ on $Y$. Then, we have the following immediate consequence of this definition and the definition of modulus of discontinuity (\Cref{def:mod-disc}).

\begin{proposition}\label{prop:mod-disc}
$\disc(\psi_\iota) = \sup\{d_X(x,x')|\,V_x\cap V_{x'}\neq \emptyset\}.$
\end{proposition}

\section{EPC constructions for the case of spheres}

We now describe a number of constructions of EPCs between spheres that we have tested exhaustively via computational methods.  As we discuss below, our extensive experimental evidence indicates  that these constructions are optimal \cite{dgh-github}.

\medskip
In what follows, for each $n\in\N$, we view the unit sphere $\Sp^n\subset \R^{n+1}$,when endowed with its geodesic distance $d_n$,  as the compact metric space $(\Sp^n,d_n)$.  Explicitly,  $$d_n(x,x') = \arccos(x\cdot x').$$

\medskip
The constructions we mention below are related to the general question of determining the precise value of $\dgh(\Sp^m,\Sp^n)$ for all $m<n$ which was considered in \cite{lim2021gromov}.

In what follows we will assume the \emph{equatorial} (isometric embedding $\iota_{2k}:\Sp^{2k}\hookrightarrow \Sp^{2k+1}$ arising from the embedding of $\R^{2k+1}\hookrightarrow \R^{2k+2}$ where $$x=(x_1,x_2,\ldots,x_{2k+1})\mapsto (x_{1},x_2,\ldots,x_{2k+1},0).$$  These embeddings, through suitable compositions, induce embeddings $\iota_{m,n}:\Sp^m\hookrightarrow\Sp^n$, for all $n\geq m$. We will henceforth identify $\Sp^m$ with its image in $\Sp^n$ via $\iota_{m,n}$. Similarly, we  consider the surjective projection maps $p_{2k+1}:\Sp^{2k+1}\backslash\{\pm e_{2k+2}\}\twoheadrightarrow \Sp^{2k}$ given by $$(x_1,x_2,\ldots ,x_{2k+1}, x_{2k+2})\mapsto \frac{(x_1,x_2,\ldots, x_{2k+1},0)}{\sqrt{x_1^2+x_2^2+\cdots +x_{2k+1}^2}}.$$

\subsection{Constructions for the case of $\Sp^1$ versus $\Sp^{n}$}

\begin{definition}
Let $n$ be any positive integer. The (projected centrally  symmetric) \emph{trigonometric moment curve} (TMC) of order $n$ is defined as follows.\footnote{For odd $n$ this definition coincides, up to a multiplicative constant, with the symmetric trigonometric moment curve considered in \cite{barvinok2008centrally}.}

When $n=2k+1$ for some $k\geq 1$, $\gamma_{2k+1}:\Sp^1\to \Sp^{2k+1}$ is given by $$t\mapsto \frac{1}{\sqrt{k+1}}\big(\cos(t),\sin(t),\cos(3t),\sin(3t),\ldots, \cos((2k+1)t),\sin((2k+1)t)\big).$$

When $n=2k$ for some $k\geq 1$, $\gamma_{2k}:\Sp^1\mapsto \Sp^{2k}$ is given by 
$$t\mapsto \frac{1}{\sqrt{k + \cos^2((2k+1)t)}}\big(\cos(t),\sin(t),\cos(3t),\sin(3t),\ldots, \cos((2k+1)t)\big).$$

\end{definition}
Note that $\gamma_{n}$ provides an embedding of $\Sp^1$ into $\Sp^{n}$ which we will henceforth refer to as a \emph{TMC embedding}. Note that $R_{\gamma_n}$ is a correspondence between $\Sp^1$ and $\Sp^{n}$ and that $\psi_{\gamma_n}$ is surjection from $\Sp^n$ to $\Sp^1$. We will refer to these correspondences as TMC-EPCs. To simplify notation, we will henceforth use the notation $R_{n}$ instead of $R_{\gamma_{n}}$, $\Pi_n$ instead of $\Pi_{\gamma_n}$ and similarly, $\psi_n$ instead of $\psi_{\gamma_n}$. See \Cref{fig:s1-sn} for an illustration.

\begin{figure}
    \centering
    \includegraphics[width=0.75\linewidth]{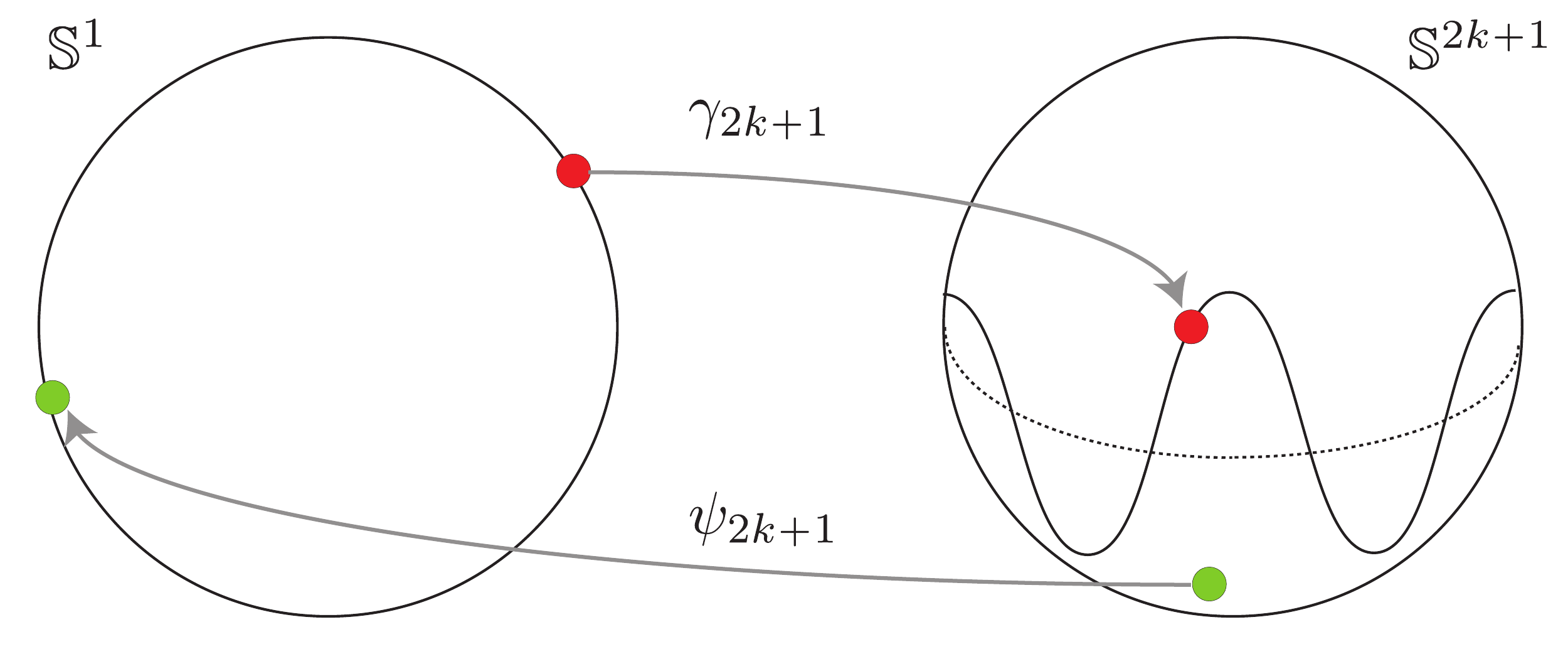}
    \caption{The maps $\gamma_{2k+1}:\Sp^1\to \Sp^{2k+1}$ and $\psi_{2k+1}:\Sp^{2k+1}\to \Sp^1$.}
    \label{fig:s1-sn}
\end{figure}

\begin{remark}\label{rem:antipodal} The following properties of the TMCs $\gamma_n$ and of the maps $\psi_n$ follow from their definitions:
\begin{enumerate}
\item $\gamma_n:\Sp^1\to \Sp^n$ is antipodal.
\item $p_{2k+1}(\gamma_{2k+1}) = \gamma_{2k}.$ 
\item The map $\psi_n:\Sp^n\to \Sp^1$ is antipodal (this follows from item 1).
\item $R_n = \mathrm{graph}(\psi_n)$ (see \Cref{rem:map-corresp}).
\end{enumerate}
\end{remark}

\Cref{thm:lower-bound} above implies that both $\dis(R_{{2k}})$ and $\dis(R_{2k+1})$ must be at least $\delta_k$. Our extensive computational experimentation strongly suggests that $R_{2k}$ and $R_{2k+1}$ are optimal.

 \begin{conjecture}\label{conj:tmc}
 $\dis(R_{{2k+1}}) = \dis(R_{{2k}}) = \frac{2\pi k}{2k+1}$ for all $k\in\N$. \end{conjecture}

Similarly, \Cref{thm:disc-lower-bound} implies that $\mathrm{disc}(\psi_{2k+1})$ and $\mathrm{disc}(\psi_{2k})$ are both bounded below by $\delta_k$ which motivates the following. 
\begin{conjecture}\label{conj:disc-psi}
$\mathrm{disc}(\psi_{2k+1})=\mathrm{disc}(\psi_{2k})=\delta_k.$
\end{conjecture}

\begin{remark}
    We point out that \Cref{conj:tmc}, if true, would imply \Cref{conj:disc-psi}. Indeed, suppose for example that $\dis(R_{2k+1}) = \delta_k$. Then, \Cref{prop:moddis}  together with the fact that $\psi_{2k+1}$ is subordinate to $R_{2k+1}$, would imply that $\delta_k \geq \mathrm{disc}(\psi_{2k+1})$. Then, the equality $\mathrm{disc}(\psi_{2k+1})=\delta_k$ would follow from \Cref{rem:antipodal} and \Cref{thm:disc-lower-bound}.

We nevertheless pose both conjectures having in mind that it might be easier to arrive at the former than  the latter. See \Cref{sec:mod-disc-3-min} and \Cref{sec:dis-3-min}.
\end{remark}

We first point out that it would be enough to establish \Cref{conj:tmc} and \Cref{conj:disc-psi} above only for $R_{2k+1}$ (respectively, $\psi_{2k+1}$) since we have the following.
\begin{proposition}\label{prop:dis2k}
It holds that $\dis(R_{{2k+1}}) \geq  \dis(R_{{2k}})$ and $\disc(\psi_{2k+1})\geq \disc(\psi_{2k})$.
\end{proposition}

\begin{proof}
These follow from the fact  that $\psi_{2k+1}|_{\Sp^{2k}} = \psi_{2k}$
where the precise copy of $\Sp^{2k}\subset \Sp^{2k+1}$ we mean is the equatorial one, as mentioned at the beginning of this section.  
\end{proof}

\begin{remark}[\faLaptopCode]
We have run extensive experimentation in relation to \Cref{conj:tmc} for the values of $k=1,2,3,4$. See our GitHub repository \cite[Function \texttt{TestDistortionRn.m}]{dgh-github}. Our results strongly suggest that $R_{2k}$ and $R_{2k+1}$ are  optimal correspondences (for the corresponding values of $k$). 
\end{remark}

\begin{example}
For example,  \Cref{conj:tmc}, if true would imply that:
\begin{itemize}
\item the distortion of the EPC between $\Sp^1$ and $\Sp^2$ induced by the curve $$\gamma_2(t) = \frac{1}{\sqrt{1+\cos^2(3t)}}\big(\cos(t),\sin(t),\cos(3t)\big)$$ is $\frac{2\pi}{3}.$

\item the distortion of the EPC between $\Sp^1$ and $\Sp^3$ induced by the curve  $$\gamma_3(t) = \frac{1}{\sqrt{2}}\big(\cos(t),\sin(t),\cos(3t),\sin(3t)\big)$$ 
is also $\frac{2\pi}{3}.$
\end{itemize}
\end{example}

\subsection{Cartoonizations and other constructions \extension}
During our work, we used the term \emph{cartoonization}  to refer to the process of  simplification of the TMC-ECPs. Some of the ideas behind this loosely defined concept involved discretizing $R_n$  or altering the nature of the curve by for example substituting it for a piece-wise geodesic path. For example, the optimal correspondence between $\Sp^1$ and $\Sp^2$ described in \cite[Appendix D]{lim2021gromov-arxiv} arose as a cartoonization of $R_2$. Similarly, the optimal correspondence constructed in \cite[Proposition 1.16]{lim2021gromov} can be regarded as a cartoonization of $R_2$.

\subsubsection{Examples of cartoonizations \extension}

\subsubsection{Another construction for the case of $\Sp^1$ versus $\Sp^2$} \label{sec:variant-R2}
Consider the following curve/embedding of $\Sp^1$ into $\Sp^2$. Let $\alpha:\Sp^1\to \Sp^2$ be given by 
$$t \mapsto \big(\cos(t)\sqrt{1-z^2(t)},\sin(t)\sqrt{1-z^2(t)},z(t)\big)$$
where $z(t):= 0.15\, \cos(3t)$.

\begin{remark}[\faLaptopCode]
This correspondence  emerged as a variant of $R_{2}$. Through our computational experiments we arrive the the following.
\end{remark}

\begin{conjecture}
    $\dis(R_\alpha)=\frac{2\pi}{3}$.
\end{conjecture}

\subsubsection{A construction for the case of $\Sp^2$ versus $\Sp^3$}\label{sec:corr-s2-s3}

As a generalization of the construction in \Cref{sec:variant-R2}, we construct a \emph{surface} $\sigma:\Sp^2\to \Sp^3$. We describe $\sigma$ in spherical coordinates on $\Sp^2$: $x,y,z:[0,2\pi)\times [0,\pi]\to \Sp^2$ where $\big(x(\phi,\theta),y(\phi,\theta),z(\phi,\theta)) := \big(\cos(\phi)\sin(\theta),\sin(\phi)\sin(\theta),\cos(\theta)\big).$ Let $$w(\phi,\theta):=\frac13\sin(\theta)\,\sin(2\theta)\,\cos(2\phi).$$

This function is proportional to the  (real) spherical harmonic $Y_{3,2}$ which has tetrahedral symmetry in the sense that its maximum value (as a function on the sphere $\Sp^2$) is attained at the vertices of a regular tetrahedron inscribed in the sphere. Then, define
$$\sigma(\phi,\theta):=\bigg(x(\phi,\theta)\sqrt{1-w^2(\phi,\theta)}, y(\phi,\theta)\sqrt{1-w^2(\phi,\theta)}, z(\phi,\theta)\sqrt{1-w^2(\phi,\theta)},w(\phi,\theta)\bigg).$$

\begin{remark}[\faLaptopCode]
The resulting correspondence $R_\sigma$ was extensively experimentally tested and it was cartoonized in the proof of \cite[Proposition 1.19]{lim2021gromov}. Theorem B in \cite{lim2021gromov} implies that $\dis(R_\sigma)\geq \zeta_2:=\arccos\left(-\frac{1}{3}\right)$.
\end{remark}

\begin{conjecture}
    $\dis(R_\sigma)=\zeta_2$.
\end{conjecture}

\subsubsection{EPCs for the case of $\Sp^m$ versus $\Sp^{n}$ \faClock}\label{sec:corr-sn-snp1}

\section{General results about TMC-EPCs}
We establish several general results about the fibers of  TMC-EPCs.  

\subsection{Some results about the fibers of $R_{2k+1}$}\label{sec:basics-fiber}
For each $t\in \Sp^1$, by $F_{2k+1}(t)$ we will denote the closed fiber 
\begin{align*}F_{2k+1}(t) &:= \overline{\{x\in \Sp^{2k+1}|\,(x,t) \in R_{2k+1}\}}\\ &= \big\{x\in\Sp^{2k+1}|d_{2k+1}(x,\gamma_{2k+1}(t))\leq d_{2k+1}(x,\gamma_{2k+1}(s))\,\forall s\in \Sp^1\big\}\\
&= \big\{x\in\Sp^{2k+1}|x \cdot\gamma_{2k+1}(t)\geq x \cdot\gamma_{2k+1}(s)\,\forall s\in \Sp^1\big\}.
\end{align*}

We will henceforth identify $\Sp^1$ with the real line modulo the equivalence relation $t\sim s$ iff $t-s = 2\pi m$ for some integer $m$.  It will be useful to introduce the following family of rotations of $\R^{2k+2}$. For $t\in \Sp^2$, let  $$T_t:=\begin{bmatrix}
M_1(t) & 0 & 0 & 0 &\cdots & 0\\
0 & M_3(t) & 0 & 0 &\cdots & 0\\
0 & 0 & M_5(t) & 0 &\cdots & 0\\
0 & 0 & 0 & M_7(t) &\cdots & 0\\
\vdots & \vdots & \vdots & \vdots &\ddots & \vdots\\
0 & 0 & 0 & 0 &\cdots & M_{2k+1}(t)\\
\end{bmatrix}$$
where, for each non-negative integer $\ell$ 
$$M_{2\ell+1}(t):=\begin{bmatrix}
\cos((2\ell+1)t) & \sin((2\ell+1)t)\\
-\sin((2\ell+1)t) & \cos((2\ell+1)t)
\end{bmatrix}.$$

These rotations have been utilized in the context of studying cyclic polytopes induced by the TMC; see \cite[Section 2]{smilansky1990bi} for the case $k=1$ and \cite[Section 3]{barvinok2008centrally} for the general case.

\begin{remark}\label{rem:rot}
Note that $T_0=\mathrm{id}$, $T_{\pm\pi} = -\mathrm{id}$ and that for all $t$ and $s$ one has 
$$\gamma_{2k+1}(t+s) =  T_t\big(\gamma_{2k+1}(s)\big).$$ It follows that, since the fibers $F_{2k+1}(t)$ are defined via a closest point map, we have that  all fibers are isometric and satisfy
$$F_{2k+1}(t) = T_t\big(F_{2k+1}(0)\big).$$
\end{remark}

Recall that a closed subset $A$ of $\Sp^n$ is said to be \emph{geodesically convex} if for any two points $p,p'\in A$ there is a unique geodesic (minimizing) geodesic connecting $p$ and $p'$ that is entirely contained within $A$. Similarly, the \emph{spherical convex hull} of $A$, denoted $\convsph(A)$, is the union of all geodesic segments with endpoints in $A$. We will reserve the notation $\conv(A)$ to denote the standard convex hull of $A$. 

Below, for $t\in \Sp^1$, by $\Sigma_t$ we will denote the $(2k+1)$-dimensional hyperplane passing through $\gamma_{2k+1}(t)$ and with normal  $\dot{\gamma}_{2k+1}(t)$: $$\Sigma_t:=\left\{p\in \R^{2k+2}|(p-\gamma_{2k+1}(t))\cdot \dot{\gamma}_{2k+1}(t)=0\right\}$$
and by $\Sp_t^{2k}$ we will denote the equator of $\Sp^{2k+1}$ obtained as its intersection with $\Sigma_t$:
$$\Sp^{2k}_t := \Sp^{2k+1}\cap \Sigma_t.$$

\begin{proposition}\label{prop:convex}
The fiber $F_{2k+1}(0)$ is a  geodesically convex subset of $\Sp^{2k+1}.$ Furthermore, $F_{2k+1}(0)\subset \Sp^{2k}_0.$
\end{proposition}

\begin{remark}
Note that $F_{2k+1}(0) \cap \Sp^{2k-1} = F_{2k-1}(0)$.
\end{remark}

The proposition follows immediately from the lemma below and from results on Voronoi partitions induced by real algebraic manifolds \cite[Proposition 8.2]{sturmfelds-mag}. Barvinok and Novik (in Section 1  of \cite{barvinok2008centrally}) and Sinn already recognized that $\gamma_{2k+1}$ is a smooth algebraic curve. This follows  from the standard facts that
\begin{itemize}
\item for each integer $m$, $\cos(m t) = \mathcal{T}_m(\cos(t))$, 
\item for odd integers $m$, $\sin(m t) = -\mathcal{T}_m(\sin(t))$,
\end{itemize}
where $\mathcal{T}_m$ is the $m$th Chebyshev polynomial (of the first kind). For example, for the case $k=1$, (the trace of) $\gamma_3$ coincides with the zero set of the collection $\{P_1,P_2,P_3,P_4\}$ of polynomials given by\footnote{Note that this is not irreducible: $P_2$ can be dropped, for example}:
\begin{align*}
P_1(x,y,z,w) &:= x^2+y^2-\tfrac{1}{2}&\\
P_2(x,y,z,w) &:= z^2+w^2-\tfrac{1}{2}&\\
P_3(x,y,z,w) &:= z-\mathcal{T}_3(x) &= z- (4x^2-3)x\\
P_4(x,y,z,w) &:= w- \mathcal{T}_3(-y) &= w-(3-4y^2)y 
\end{align*}

\begin{lemma}\label{lemma:algebraic}
For each $k\geq 1$, the symmetric trigonometric moment curve $\gamma_{2k+1}:\Sp^1\to\Sp^{2k+1}$ can be modeled as a smooth real algebraic curve.
\end{lemma}

\begin{proof}[Proof of \Cref{prop:convex}]
The  fiber $F_{2k+1}(0)$ is the intersection $V_{2k+1}(0)\cap \Sp^{2k+1}$ of the closed Voronoi cell $V_{2k+1}(0)\subset \R^{2k+2}$ corresponding to $\gamma_{2k+1}(0)$ in the Voronoi tiling of $\R^{2k+2}$ induced by the curve $\gamma_{2k+1}.$
By the lemma above,  $\gamma_{2k+1}$ is a smooth real algebraic curve so that, by \cite[Proposition 8.2]{sturmfelds-mag}, $V_{2k+1}(0)$ is convex.  It is clear that $V_{2k+1}(0)$ is a convex cone containing the origin so that  the claim follows from \cite[Proposition 2 and Remark 1]{ferreira2013projections}. 

The second claim that $F_{2k+1}(0)$ is contained in $\Sp^{2k}_0$ can be explained as follows. One first recalls that 
$$F_{2k+1}(0)=\{x\in \Sp^{2k+1}|\,x\cdot \gamma_{2k+1}(0)\geq x\cdot \gamma_{2k+1}(t)\,\,\forall t\in \Sp^1\}.$$ So that $x\in F_{2k+1}(0)$ then means that the function $t\mapsto f(t;x):=x\cdot \gamma_{2k+1}(t)$ has a global maximum at $t=0$ which implies that $x\cdot \dot{\gamma}_{2k+1}(0) = 0$ whence the claim.
\end{proof}

We now state the following  relationship between the modulus of discontinuity of $\psi_{2k+1}$ and the fibers $F_{2k+1}(\cdot)$ for later use. Via  \Cref{prop:mod-disc} and from the definition of   $F_{2k+1}(0)$, we immediately obtain the following.
\begin{corollary}\label{coro:disc-voro}
    $\disc(\psi_{2k+1})=\max\{t\in[0,\pi]|\,\partial F_{2k+1}(0)\cap \partial F_{2k+1}(t) \neq \emptyset\}.$
\end{corollary}

\subsection{A simplification of the calculation of the distortion of $R_{2k+1}$}\label{sec:simple-conditions}
We exploit  symmetries of $R_{2k+1}$ in order to simplify the determination of its distortion.
\begin{remark}\label{rem:dis} Note that, in order to estimate/calculate  the distortion  of $R_{2k+1}$, it suffices to consider, for all $q,q'\in F_{2k+1}(0)$ and $t\in\Sp^1$, the quantity
$$\delta_t(q,q'):=\big|d_{2k+1}(q,T_t(q'))-d_1(0,t)\big|.$$
In other words, 
one has $$\dis(R_{2k+1}) = \max_{\substack{q,q'\in F_{2k+1}(0)\\ t\in \Sp^1}} \delta_t(q,q').$$
The claim follows from \Cref{rem:rot} and the following calculation
\begin{align*}\dis(R_{2k+1}) &= \max_{s,t\in\Sp^1}\max_{\substack{q\in F_{2k+1}(t)\\q'\in F_{2k+1}(s)}}\big|d_{2k+1}(q,q') - d_1(s,t)\big|\\& = \max_{\substack{q,q'\in F_{2k+1}(0)}}\max_{s,t\in\Sp^1}\big|d_{2k+1}(T_tq,T_sq') - d_1(s,t)\big|\\
&=\max_{\substack{q,q'\in F_{2k+1}(0)}}\max_{s,t\in\Sp^1}\big|d_{2k+1}(q,T_{s-t}q') - d_1(0,s-t)\big|.
\end{align*}
\end{remark}

For $\delta\in[0,\pi$], consider the following four properties

\begin{align*}
    A'(\delta): & & d_{2k+1}(q,T_t q') \leq d_1(0,t) + \delta && \forall \, q,q'\in F_{2k+1}(0),\, t \in \Sp^1\\
    A(\delta): & & d_{2k+1}(q,T_t q') \leq d_1(0,t) + \delta && \forall \, q,q'\in F_{2k+1}(0),\, |t| \leq \pi-\delta\\
    &&&&\\
    B'(\delta): & &   d_1(0,t)  \leq d_{2k+1}(q,T_t q') + \delta && \forall \, q,q'\in F_{2k+1}(0),\, t \in \Sp^1\\
    B(\delta): & &   d_1(0,t)  \leq d_{2k+1}(q,T_t q') + \delta && \forall \, q,q'\in F_{2k+1}(0),\, |t| \in [\delta, \pi]
\end{align*}
The following proposition simplifies the task of checking whether $\dis(R_{2k+1})\leq \delta$ to checking whether $B(\delta)$ holds.

\begin{proposition}\label{prop:eq-dis}
   For each $\delta \in [0,\pi]$ we have:
   \begin{itemize}
\item[(a)] $\dis(R_{2k+1})\leq \delta$ $\Longleftrightarrow$ $A'(\delta)$ and $B'(\delta)$ hold.
\item[(b)] $A(\delta) \Longleftrightarrow A'(\delta)$.
\item[(c)] $B(\delta) \Longleftrightarrow B'(\delta)$.
\item[(d)] $B(\delta)\Longrightarrow A(\delta).$
   \end{itemize}
\end{proposition}

\begin{proof}
    (a) follows from \Cref{rem:dis}.  (b), (c) are clear. Let's prove (d). Assume $B(\delta)$ holds so that, by (c), $B'(\delta)$ also holds. Pick $t$ such that $|t|\leq \pi-\delta$ and fix $q,q'\in F_{2k+1}(0).$ WLOG we can assume that $t\in[0,\pi-\delta]$. Notice that, since $T_{t-\pi}(q') = T_{-\pi}T_t(q')=-T_t(q')$, we have
    $$d_{2k+1}(q,T_tq')+d_{2k+1}(q,T_{t-\pi}q')=\pi.$$

    Since $B'(\delta)$ holds, we have that $d_{2k+1}(q,T_{t-\pi}q') \geq d_1(0,t-\pi)-\delta$ so that 
     \begin{align*}
        d_{2k+1}(q,T_tq') &= \pi-d_{2k+1}(q,T_{t-\pi}q')\\
        &\leq \pi+\delta -d_1(0,t-\pi)\\
        &= \pi+\delta - (\pi-t) \\
        &=t+\delta\\
        &=d_1(0,t) + \pi
    \end{align*}
    which implies that $A(\delta)$ holds.
\end{proof}

One additional simplification consists of checking $B(\delta)$ only for points $a,a'\in \partial F_{2k+1}(0).$ Define the condition
\begin{align}\label{eq:Baster}
B^*(\delta): & &   d_1(0,t)  \leq d_{2k+1}(q,T_t q') + \delta && \forall \, q,q'\in \partial F_{2k+1}(0),\, |t| \in [\delta, \pi] \tag{$\star$}
\end{align}
\begin{proposition}\label{prop:B}
   $B(\delta)\Longleftrightarrow B^*(\delta).$
\end{proposition}
\begin{proof}
We only need to prove that $B^*(\delta)\Longrightarrow B(\delta).$ Notice that $B(\delta)$ is equivalent to 

\begin{align*}
B"(\delta): & &   d_1(0,t)  \leq \min\{d_{2k+1}(q,T_t q');\,q,q'\in F_{2k+1}(0)\} + \delta && \forall  |t| \in [\delta, \pi]
\end{align*}
 and that 
 $$\min\{d_{2k+1}(q,T_t q');\,q,q'\in F_{2k+1}(0)\} = \min\{d_{2k+1}(q,q');\,q\in F_{2k+1}(0)\,\text{and}\,q'\in F_{2k+1}(t)\}.$$
  These imply the claim since $F_{2k+1}(0)$ (and therefore each fiber) is geodesically convex\footnote{For $t\neq 0$, the  fiber $F_{2k+1}(t)$ can intersect $F_{2k+1}(0)$ only at boundary points.} so that the minimum on the RHS is attained at boundary points.
\end{proof}
\begin{framed} 
As a consequence of \Cref{prop:eq-dis,prop:B,prop:dis2k}, we see that, in order to establish \Cref{conj:tmc}, it suffices to establish $B^*(\delta)$ for $\delta = \delta_k=\frac{2\pi k}{2k+1}.$
 \end{framed}


\subsection{Other results}
In this section we introduce a convenient parameterization of odd-dimensional spheres and also explore some preliminary results.
\subsubsection{Hopf coordinates on $\Sp^{2k+1}$}\label{sec:hopf}
Hopf coordinates provide a  well known coodinatization of $\Sp^3\subset \R^4$ 
so that a point on  $\Sp^3$ can be (uniquely) written  as  
$$q(\theta_1,\theta_1,\zeta) :=\big(\cos(\theta_1)\cos(\zeta),\sin(\theta_1)\cos(\zeta),\cos(\theta_2)\sin(\zeta),\sin(\theta_2)\sin(\zeta)\big),$$
for $\theta_1,\theta_2\in[0,2\pi)$ and $\zeta\in[0,\frac{\pi}{2}].$ These coordinates are closely linked to the (topological) fact that $\Sp^3$ is homeomorphic to the topological join $\Sp^1* \Sp^1$ of two copies of the circle; see \Cref{fig:hopf-s3}.

\begin{figure}
\centering
\includegraphics[width=\linewidth]{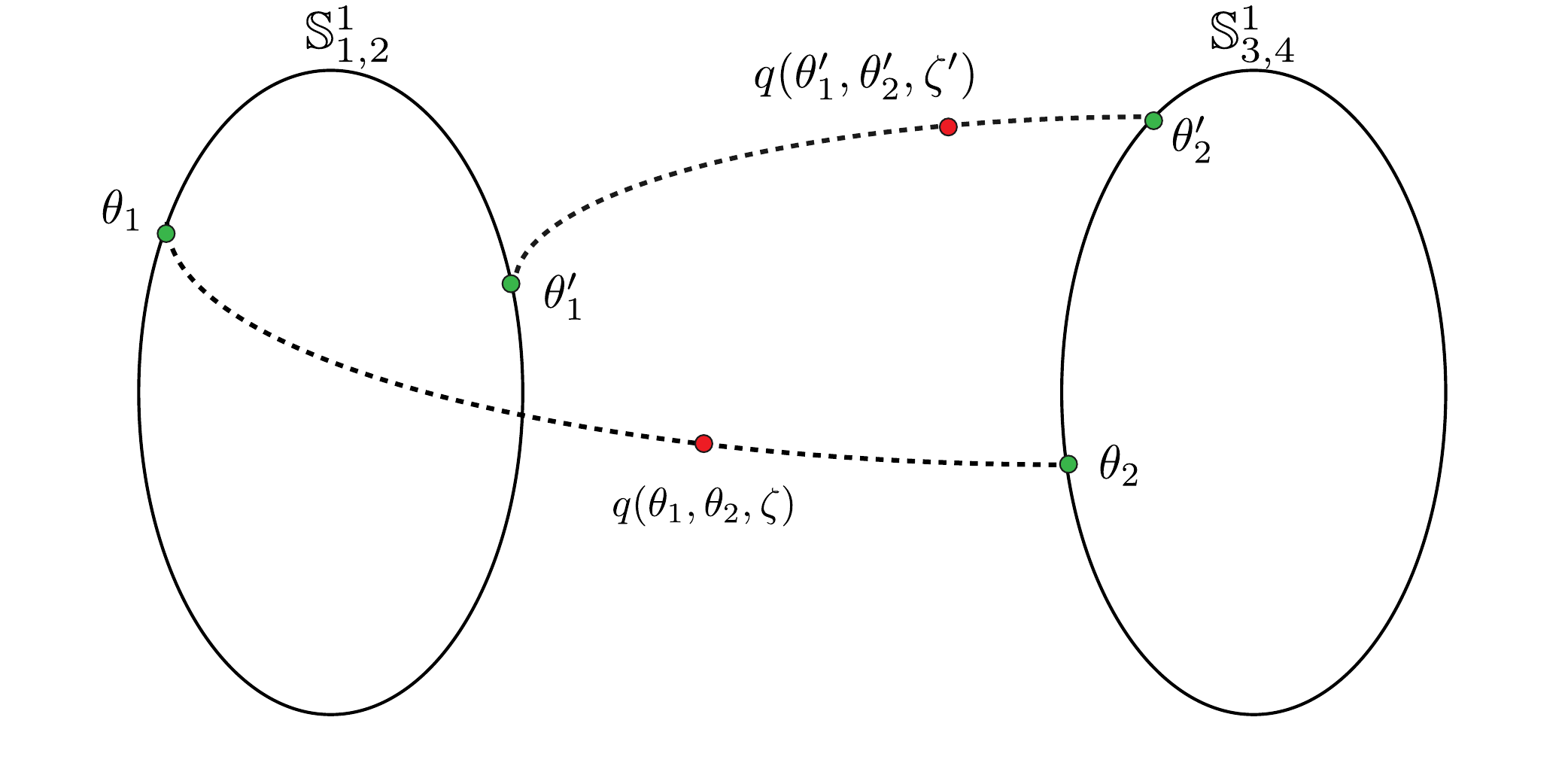}
\caption{Hopf coordinates on $\Sp^3$. See \Cref{sec:hopf}.} \label{fig:hopf-s3}
\end{figure}

In general, since $\Sp^{2k+1}$ is homeomorphic to the $(k+1)$-fold join $\Sp^1 *\cdots * \Sp^1$ ($k+1$ times), one would suspect that it is possible to generalize these coordinates to higher dimensional odd dimensional spheres. This is indeed the case; one can induce such coordinates as follows. 

Firstly, consider $k$ copies $\Sp^1$ labeled as $\Sp_{1,2}^1,\Sp^1_{3,4},\ldots,\Sp^1_{2i+1,2i+2},\ldots,\Sp^1_{2k+1,2k}$ such that  the copy $\Sp^1_{2i+1,2i+2}$ lies in $\mathrm{span}(e_{2i+1},e_{2i+2})$ where $\{e_1,e_2,\ldots,e_{2k+2}\}$ is canonical basis of $\R^{2k+2}$. Then, introduce $(\theta_1,\ldots,\theta_{k+1}) \in[0,2\pi)^{k+1}$ such that:

\begin{itemize}
\item $p_{1}\in \Sp^1_{1,2}$ has coordinates $(\cos(\theta_1),\sin(\theta_1),0,0,\ldots,0)$ for $\theta_1\in[0,2\pi),$
\item $p_{2}\in \Sp^1_{3,4}$ has coordinates $(0,0,\cos(\theta_2),\sin(\theta_2),0,0,\ldots,0)$ for $\theta_2\in[0,2\pi),$
\item $p_{i}\in \Sp^1_{2i-1,2i}$ has coordinates $(0,0,0,0,\ldots,0,0,\cos(\theta_i),\sin(\theta_i),0,0,\ldots,0)$ for $\theta_i\in[0,2\pi),$
\item $p_{k+1}\in \Sp^1_{2k+1,2k+2}$ has coordinates $(0,0,\ldots,0,0\cos(\theta_k),\sin(\theta_k))$ for $\theta_{k+1}\in[0,2\pi).$
\end{itemize}

Secondly, consider $(a_1,a_2,\ldots,a_i,\ldots,a_{k+1}) \in [0,1]^{k+1}$ such that $\sum_i a_i^2 =1$.

\medskip
Then, the Hopf coordinates of a point in $\Sp^{2k+1}$ will be denoted by $$q(p_1,\ldots,p_{k+1};a_1,\ldots,a_{k+1})$$ or alternatively as 
$$q(\theta_1,\ldots,\theta_{k+1};a_1,\ldots,a_{k+1}).$$

\begin{remark}
The subsets $C(a_1,\ldots,a_{k+1})$ of $\Sp^{2k+1}$ obtained via the above parametrization for fixed $a_1,\ldots,a_{k+1}$ are analogues of the Clifford tori in $\Sp^3$.
\end{remark}

Directly from the definition of the  rotation $T_t$ in \Cref{sec:basics-fiber} and by the description of Hopf coordinates we obtain the following (using the notation established above).
\begin{corollary}\label{coro:preserv-tori}
For all $t\in \Sp^1$, 
\begin{align*}T_t\big(q(\theta_1,\ldots,\theta_i,\ldots,\theta_{k+1};a_1,\ldots,a_i\ldots,a_{k+1}\big) &=\\ q(\theta_1+t,\ldots,\theta_i+(2i-1)t,\ldots,\theta_k+(2k+1)t;a_1,\ldots,a_i,\ldots,a_k).\end{align*}
In particular, $T_t$ does not affect the values of the coordinates $a_1,\ldots,a_{k+1}$.
\end{corollary}

Using Hopf coordinates, for a given point $q = q(\theta_1,\ldots,\theta_{k+1};a_1,\ldots,a_{k+1}) \in \Sp^{2k+1}$, we then define the following function $P_{2k+1}(\cdot;q):(-\pi,\pi]\to\R$ given by
\begin{equation}\label{eq:P}
    P_{2k+1}(t;q):= q\cdot \gamma_{2k+1}(t) = \frac{1}{\sqrt{k+1}}\sum_{\ell=0}^{k}a_\ell \,\cos((2\ell+1)t-\theta_{\ell+1}),\end{equation}
    so that $d_{2k+1}(q,\gamma_{2k+1}(t)) = \arccos\big(P_{2k+1}(t;q)\big).$

\subsubsection{Some properties of $\gamma_{2k+1}$ and the TMC-EPC}

The following conjecture is suggested by the overall goal of proving that the distortion of the TMC-EPC is minimal (see \Cref{conj:tmc}).  Recall the discussion in \Cref{sec:interp-EPC} which yields a lower bound for an EPC via the distortion of the embedding map and the covering radius of the projection $\Pi_{2k+1}:\Sp^{2k+1}\to \gamma_{2k+1}(\Sp^1)$:
\begin{equation}\label{eq:motiv-TMC-EPC}\dis(\psi_{2k+1})\geq \max\big(\dis(\gamma_{2k+1}),\rho(\Pi_{2k+1})\big).\end{equation}

\begin{conjecture}[covering radius]\label{conj:cov}
$$\rho(\Pi_{2k+1}) = \max_{q\in \Sp^{2k+1}} \min_{t\in\Sp^1}d_{2k+1}(q,\gamma_{2k+1}(t)) \leq \frac{\delta_k}{2} = \frac{\pi k}{2k+1}.$$
\end{conjecture}

\begin{remark}\label{rem:cov-radius-pi-by-2}
Note that the LHS $\rho(\Pi_{2k+1})$ is always bounded above by $\tfrac{\pi}{2}$. Indeed, if $q$ is any point on $\Sp^{2k+1}$, then
$$d_{2k+1}(q,\gamma_{2k+1}(t))+d_{2k+1}(q,-\gamma_{2k+1}(t)) = \pi$$
for every $t\in\Sp^1$. But, since $-\gamma_{2k+1}(t) = \gamma_{2k+1}(-t)$, the point $q$ is at distance at most $\tfrac{\pi}{2}$ from the set $\{\gamma_{2k+1}(t),\gamma_{2k+1}(-t)\}.$
\end{remark}

\begin{remark}
    Experimentally, we've obtained the estimate $\rho(\Pi_3)\approx 0.9229.$ Also, we've obtained the estimate $\dis(\gamma_3)\simeq 0.8128$ (see \Cref{table:dis-gamma}). Using these,  one obtains 
    $$\dgh(\Sp^1,\Sp^3)\leq \tfrac{1}{2}\dis(\gamma_3)+\dha(\Sp^3,\gamma_3)\approx  1.3293$$
    which is larger than the desired upper bound $\tfrac{\pi}{3}\approx 1.0472.$
\end{remark}

In \Cref{sec:cov} we will suggest a strategy for approaching  \Cref{conj:cov} via the combinatorial structure of the so called Barvinok-Novik polytope  (see \Cref{thm:cov-3}). 

\medskip

Also motivated by \Cref{eq:motiv-TMC-EPC}, the following proposition establishes that, when restricted to the image of $\gamma_{2k+1}\subset \Sp^{2k+1}$,  the distortion of $\psi_{2k+1}$, does not exceed $\delta_k$. This can be interpreted as a partial result towards \Cref{conj:tmc}.

\medskip
Note that, by definition of $\psi_{2k+1}$ one has $F_{2k+1}(t)\cap \Ima(\gamma_{2k+1}) = \{\gamma_{2k+1}(t)\}$ for each $t\in \Sp^1$.

\begin{proposition}[distortion of $R_{2k+1}$ restricted to TMC]\label{prop:restrict}
For all $s,t \in \Sp^1$ one has 
$$\big|d_{2k+1}\big(\gamma_{2k+1}(s),\gamma_{2k+1}(t)\big)-d_1(s,t)\big|\leq \delta_k.$$
In other words, 
$$\dis(\gamma_{2k+1})\leq \delta_k.$$
\end{proposition}

\begin{remark}\label{rem:dis-gamma}
The upper bound $\delta_k$ given above for $\dis(\gamma_{2k+1})$ is not tight. For example, we can prove $\dis(\gamma_3)\leq \tfrac{\pi}{3} = \tfrac{\delta_1}{2}$ \extension.  However, experimentally, we have verified that it is not true that $\dis(\gamma_{2k+1})\leq \tfrac{\delta_k}{2}$ for all $k$; see  \Cref{table:dis-gamma} below, where this is violated for  $k=5,6$. 
\begin{table}[h!]
\centering
 \begin{tabular}{c| c c c} 
 $k$ & $\dis(\gamma_{2k+1})$ & $\tfrac{\delta_k}{2}$ & $\delta_k$ \\ [0.5ex] 
 \hline\hline
 1 & 0.8128 & 1.0472 & 2.0944 \\ 
 2 & 1.1114 & 1.2566 & 2.5133 \\
 3 & 1.2694 & 1.3464 & 2.6928 \\
 4 & 1.3676 & 1.3963 & 2.7925 \\
5 & 1.4345 & 1.4280 & 2.8560 \\
6 & 1.4831 & 1.4500 & 2.8999 \\
 \end{tabular}
 \caption{Experimentally obtained values of $\dis(\gamma_{2k+1})$ for several values of $k$. See \Cref{rem:dis-gamma}.}
 \label{table:dis-gamma}
 \end{table}
 Our experiments (see \Cref{fig:dis-gamma}) also strongly suggest the following.
\end{remark}

\begin{conjecture}
 $k\mapsto \dis(\gamma_{2k+1})$ is monotonically increasing with $k$.\footnote{Our experiments  indicate that  $\dis(\gamma_{2k+1})$ never exceeds $\approx 1.77$.}
\end{conjecture}

\begin{figure}
\centering
\includegraphics[width=0.5\linewidth]{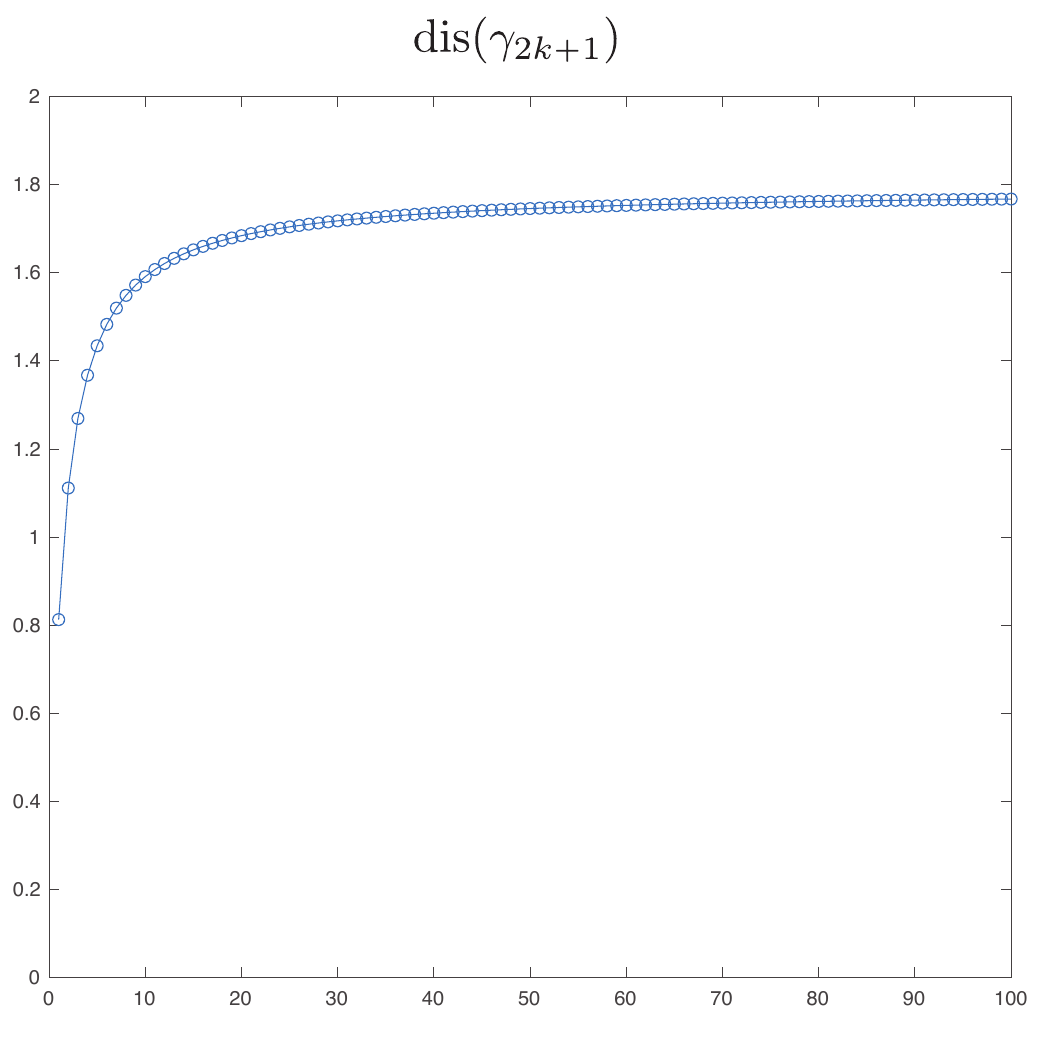}
\caption{$\dis(\gamma_{2k+1})$ as a function of $k$. See \Cref{rem:dis-gamma}} \label{fig:dis-gamma}
\end{figure}

The following function appears often  in our considerations. For $t\in \Sp^{2k+1}$ let 
\begin{equation}\label{eq:hk}h_k(t):= \gamma_{2k+1}(0)\cdot \gamma_{2k+1}(t) = \frac{1}{k+1}\sum_{\ell=0}^k \cos((2\ell+1)t).\end{equation}
Note that for $s,t\in \Sp^{2k+1}$,
$$d_{2k+1}(\gamma_{2k+1}(s),\gamma_{2k+1}(t))\big) = \arccos\big(h_k(s-t)\big).$$
Here we collect a number of useful properties of this function

\begin{proposition}\label{prop:props-hk}
\begin{enumerate}
    \item $h_k(t) =  h_k(-t)$ for all $t\in\Sp^1$.
    \item For $t\in[0,\pi]$ we have $$h_k(t)=\begin{cases}
    \frac{1}{2(k+1)}\frac{\sin\big(2(k+1)t\big)}{\sin(t)} & \text{for $t\neq 0,\pi$}\\
    1 & \text{for $t=0$}\\
    -1 & \text{for $t=\pi$}.
    \end{cases}$$
    \item For all $t\in \Sp^1$,
    $$\min_{\ell=0,\ldots,k} \cos((2\ell+1)t)\leq h_k(t)\leq \max_{\ell=0,\ldots,k} \cos((2\ell+1)t).$$
    \item For $t\in[0,\delta_k]$, 
    $$h_k(t)\geq \cos(\delta_k).$$

\end{enumerate}
   
\end{proposition}

\medskip

The following lemma is elementary but useful.
\begin{lemma}\label{lem:simple}
Let $k$ be a non-negative integer. Then, for all $t,s\in\Sp^1$, we have
$$\big|d_1(s,t) - d_1((2k+1)\,s ,(2k+1)\,t)\big|\leq \delta_k.$$ 
\end{lemma}

\begin{proof}[Proof of \Cref{prop:restrict}]
The statement is equivalent to the condition
$$\big|d_{2k+1}\big(\gamma_{2k+1}(0),\gamma_{2k+1}(t)\big)-d_1(0,t)\big|\leq \delta_k$$
for all $t\in \Sp^1.$  
Note that 
$$\cos\big(d_{2k+1}(\gamma_{2k+1}(0),\gamma_{2k+1}(t))\big) =h_k(t).$$
Therefore, by item 3 of \Cref{prop:props-hk} ,
$$\max_{\ell=0,\ldots,k} \arccos(\cos((2\ell+1)t))\geq \arccos(h_k(t))\geq \max_{\ell=0,\ldots,k} \arccos(\cos((2\ell+1)t)),$$
which is equivalent to 
$$\max_{\ell=0,\ldots,k} d_1(0,(2\ell+1)t))\geq d_{2k+1}(\gamma_{2k+1}(0),\gamma_{2k+1}(t))\geq \min_{\ell=0,\ldots,k} d_1(0,(2\ell+1)t).$$
The conclusion now follows from \Cref{lem:simple} and the fact that $\delta_\ell \leq \delta_k$ whenever $\ell\leq k$. For example, have have that the RHS above can be bounded as follows
\begin{align*}\min_{\ell=0,\ldots,k} d_1(0,(2\ell+1)t) &\geq  \min_{\ell=0,\ldots,k} \big(d_1(0,t)-\delta_\ell\big)\\ &= d_1(0,t) - \max_{\ell=0,\ldots,k}\delta_\ell\\ &= d_1(0,t) - \delta_k.\end{align*}

\end{proof}

\begin{remark} Via \Cref{prop:props-hk}, we have the following more or less explicit expression 
$$\dis(\gamma_{2k+1})=\sup_{t\in(0,\pi)}\left| \arccos(h_k(t))-t\right|=\sup_{t\in(0,\pi)}\left|\arccos\left(\frac{\sin\big(2(k+1)t\big)}{2(k+1)\sin(t)}\right)-t\right|.$$

Note that for $t=t_k:= \tfrac{\pi}{2(k+1)}$ one has $h_k(t_k)=0,$
which implies that 
$$\dis(\gamma_{2k+1})  \geq  \tfrac{\pi}{2} \tfrac{k}{k+1}.$$
It would be interesting to compute a more precise estimate of $\dis(\gamma_{2k+1})$ \extension.

\end{remark}

\begin{proof}[Proof of \Cref{prop:props-hk}]
    Item 1 is obvious. Item 2 follows via standard trigonometric manipulations (more precisely, through one of the so called Lagrange trigonometric identities). Item 3 is also obvious. 
\end{proof}

\begin{proof}[Proof of \Cref{lem:simple}]
It suffices to to prove the following two inequalities:

\begin{equation}\label{eq:plus}
d_1((2k+1)t,0) \leq d_1(t,0) + \delta_k\,\,\mbox{for $t\in[0,\pi-\delta_k]$}
\end{equation}
and
\begin{equation}\label{eq:minus}
d_1(t,0) \leq d_1((2k+1)t,0) + \delta_k\,\,\mbox{for $t\in[\delta_k,\pi]$}.
\end{equation}

Let's verify \Cref{eq:plus} first. Note that, since $\pi-\delta_k = \tfrac{\pi}{2k+1}$, 
for $t\in[0,\pi-\delta_k]$ one has $(2k+1)t \in[0,\pi]$ and that $(2k+1)t\geq t$. Therefore (see \Cref{fig:lemma-circ}), 
$$d_1((2k+1),t,0) = d_1((2k+1)t,t)+d_1(t,0)= 2kt + d_1(t,0)\leq \tfrac{2k\pi}{2k+1} +d_1(t,0) = \delta_k +d_1(t,0).$$

Now we tackle \Cref{eq:minus}. In that case, write $t=\delta_k + \nu$ for $\nu\in[0,\pi-\delta_k]$. Then,  $(2k+1)t = 2k\pi + (2k+1)\nu$  equals $(2k+1)\nu$ modulo $2\pi$. Furthermore,  given that $\nu\in[0,\pi-\delta_k]$, we have  $(2k+1)\nu \in [0,\pi]$ and $(2k+1)\nu\geq\nu.$ Hence,
$d_1(t,0) = \delta_k+\nu$ and $d_1((2k+1)t,0) = (2k+1)\nu$ so that 
$$d_1(t,0) = \nu+\delta_k \leq (2k+1)\nu+\delta_k = d_1((2k+1)t,0)+\delta_k.$$
\end{proof}

\begin{figure}
\centering
\includegraphics[width=\linewidth]{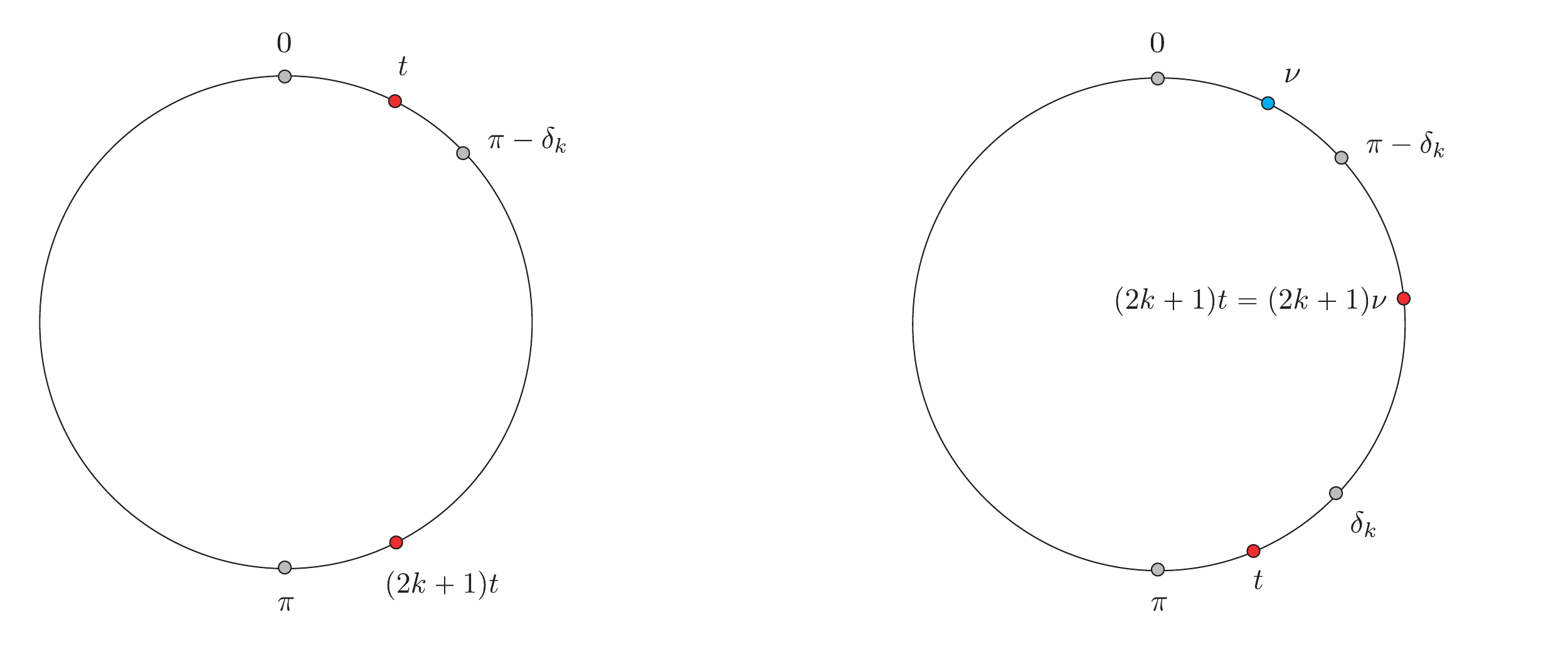}
\caption{An illustration for the proof of \Cref{lem:simple}. Left: the configuration corresponding to \Cref{eq:plus}. Right: the configuration corresponding to \Cref{eq:minus}.}\label{fig:lemma-circ}
\end{figure}

\section{The case of $\Sp^1$ versus $\Sp^3$}

We will use Hopf coordinates $(\theta_1,\theta_2,\zeta)$ on $\Sp^3$ so that a generic point on $\Sp^3\subset \R^4$ can be written  as  
$$q(\theta_1,\theta_1,\zeta) :=\big(\cos(\theta_1)\cos(\zeta),\sin(\theta_1)\cos(\zeta),\cos(\theta_2)\sin(\zeta),\sin(\theta_2)\sin(\zeta)\big),$$
for $\theta_1,\theta_2\in[-\pi,\pi)$ and $\zeta\in[0,\frac{\pi}{2}].$ See \Cref{fig:hopf-s3} for an illustration.

\subsection{Characterization of the fiber $F_3(0)$}
We have a complete characterization of the fibers of $R_3$. 
\begin{theorem}\label{thm:fiber-3} We have
$$\partial{F_3(0)} = \left\{q(\theta,3\theta+\pi,\zeta(\theta))|\theta \in\left[-\tfrac{\pi}{3},\tfrac{\pi}{3}\right)\right\}$$
where \begin{equation}\label{eq:def-eta-theta}\zeta(\theta):=\arccot\big(3(3-4\sin^2(\theta))\big).\end{equation}
Furthermore, $F_3(0) = \convsph(\partial F_3(0))$, the geodesic convex hull of $\partial F_3(0)$. 
\end{theorem}

See \Cref{fig:fiber-3} for a visualization of $F_3(0)$. 

\begin{figure}
\centering
\includegraphics[width = \linewidth]{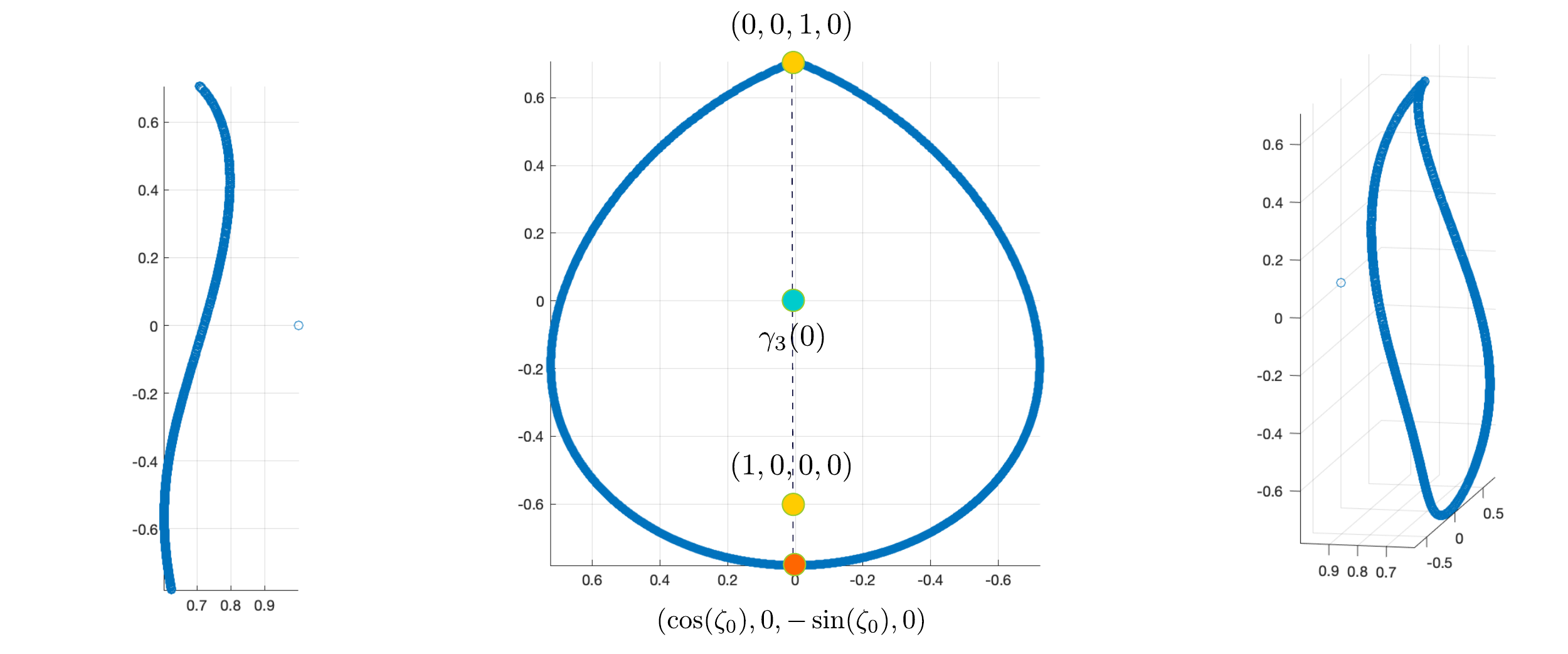}
\caption{Three views of the boundary of the fiber $F_3(0)$; see \Cref{thm:fiber-3}. The center figure shows the the points $\Sp^1_{3,4}\ni(0,0,1,0) = \bar{q}(\pm\tfrac{\pi}{3})$,    $(\cos(\zeta_0),0,-\sin(\zeta_0),0) = \bar{q}(0)$ (where $\zeta_0:=\arccot(9)$) both lying on $\partial F_3(0)$ together with  the points $\gamma_3(0) = \tfrac{1}{\sqrt{2}}(1,0,1,0)$ and $\Sp^1_{1,2}\ni (1,0,0,0)$ both lying in the interior of $F_3(0)$; see \Cref{fig:hopf-s3}.}
\label{fig:fiber-3}
\end{figure}

\begin{remark}\label{rem:boundary-3}
As a consequence of \Cref{thm:fiber-3}, a generic point on  $\partial F_3(0)$ has the following Hopf coordinates:
\begin{align*}\bar{q}(\theta) &:= q(\theta,3\theta+\pi,\zeta(\theta))\\ &=\big(\cos(\theta)\cos(\zeta(\theta)),\sin(\theta)\cos(\zeta(\theta)),\cos(3\theta+\pi)\sin(\zeta(\theta)),\sin(3\theta+\pi)\sin(\zeta(\theta))\big)\\
&=\big(\cos(\theta)\cos(\zeta(\theta)),\sin(\theta)\cos(\zeta(\theta)),-\cos(3\theta)\sin(\zeta(\theta)),-\sin(3\theta)\sin(\zeta(\theta))\big)
\end{align*}
for $\theta\in[-\tfrac{\pi}{3},\tfrac{\pi}{3}).$
\end{remark}

\begin{remark}\label{rem:cov-3-tighter}
Using the parametrization of $\partial F_3(0)$ given by \Cref{thm:fiber-3} we can consider the function
 $$[-\tfrac{\pi}{3},\tfrac{\pi}{3}]\ni \theta\mapsto \rho_3(\theta):=d_3(\bar{q}(\theta), \gamma_3(0))=\arccos\left(\bar{q}(\theta)\cdot\gamma_3(0)\right)$$ and find its maximum value. A plot of $\rho_3(\theta)$ for $\theta\in[-\tfrac{\pi}{3},\tfrac{\pi}{3})$ is shown in \Cref{fig:d3}. Numerically, we determined that the maximum value of $\rho_3$ is $\approx 0.9232$, which is strictly smaller than the upper bound $\tfrac{\pi}{3}$ implied by \Cref{thm:cov-3}.  The minimum value is  $\approx 0.6476$.
\end{remark}

\begin{figure}
    \centering
    \includegraphics[width=0.45\linewidth]{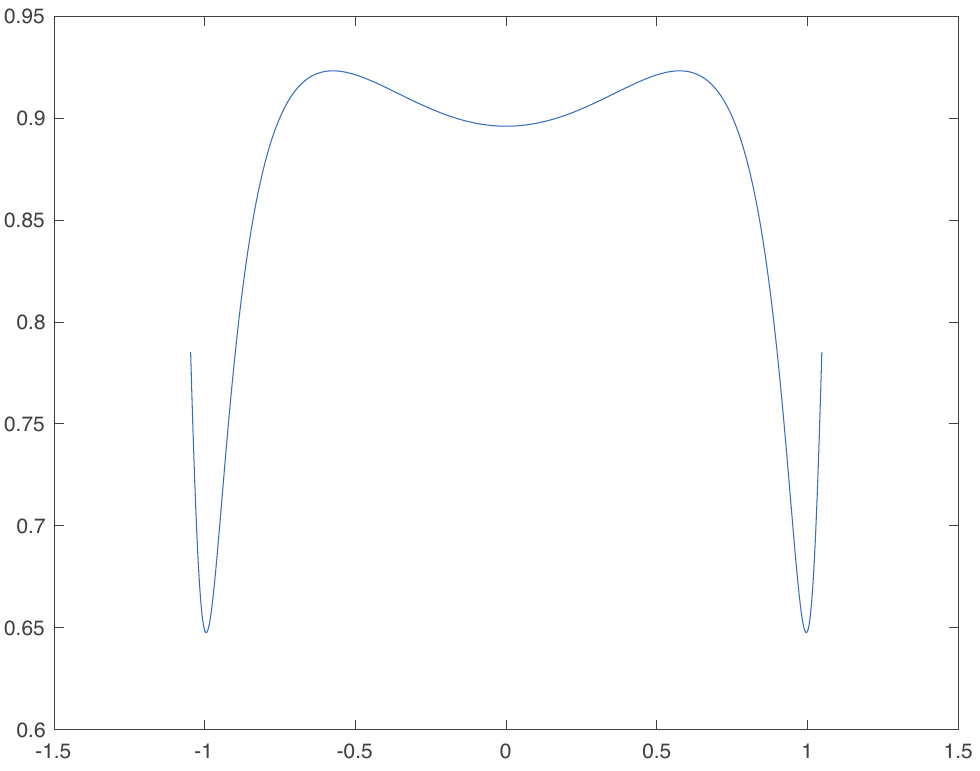}
    \includegraphics[width=0.45\linewidth]{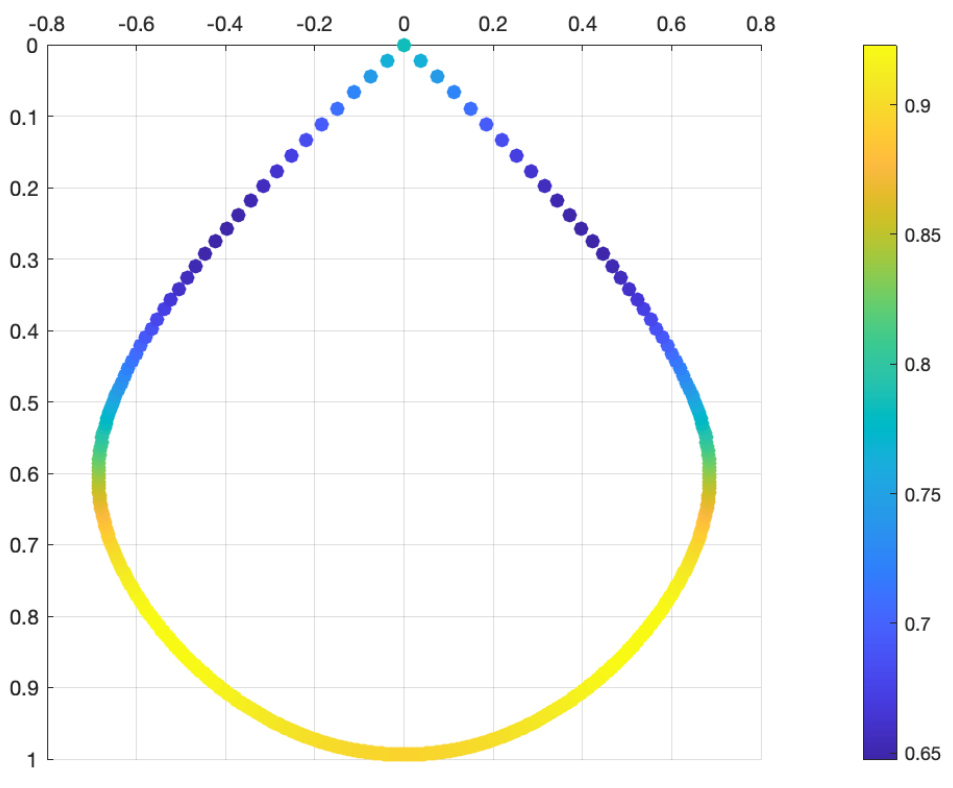}
    \caption{\textbf{Left:} Plot of  $\rho_3(\theta): = d_3(\bar{q}(\theta),\gamma_3(0))$ for $\theta\in[-\tfrac{\pi}{3},\tfrac{\pi}{3})$; see \Cref{rem:cov-3-tighter}. The maximum value of $\rho_3$ is approximately $0.9232$ whereas the minimum value is  $\approx 0.6476$. \textbf{Right:} Plot of  $\partial F_3(0)$ colored by values of $\rho_3$.}
    \label{fig:d3}
\end{figure}

\begin{lemma}\label{lem:max}
Let $\zeta\in[0,\tfrac{\pi}{2}]$ and $\theta\in[-\pi,\pi)$. Consider the trigonometric function
$$P_{\zeta,\theta}:[-\pi,\pi)\to \R$$ defined by $$P_{\zeta,\theta}(t):=\cos(\zeta)\cos(t) + \sin(\zeta) \cos(3t-\theta)$$ 
and let $M(P_{\zeta,\theta})$ denote the  set of all global maxima points of $P_{\zeta,\theta}$. Then, 
\begin{itemize}
\item $M(P_{\zeta,\theta})\subset [-\tfrac{\pi}{3},\tfrac{\pi}{3}]$ whenever $\zeta\in[0,\tfrac{\pi}{2})$.
\item $|M(P_{\zeta,\theta})|= 2$  precisely when  $\theta=\pm \pi$ and $\zeta\in(\zeta_0,\tfrac{\pi}{2})$, where $\zeta_0:=\arccot(9)$. 
\item $|M(P_{\zeta,\theta})|=3$ precisely when $\zeta = \tfrac{\pi}{2}$.
\item $|M(P_{\zeta,\theta})|=1$ in all other cases.
\end{itemize}
\end{lemma}

\begin{figure}[h]
\centering\includegraphics[width=0.45\textwidth]{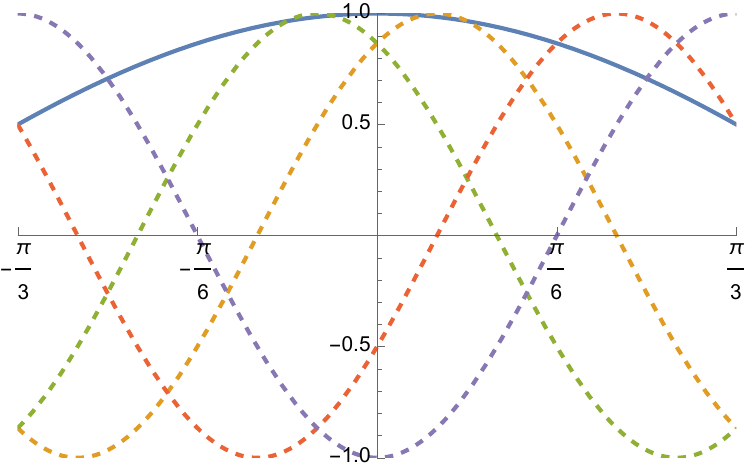}
\includegraphics[width=0.45\textwidth]{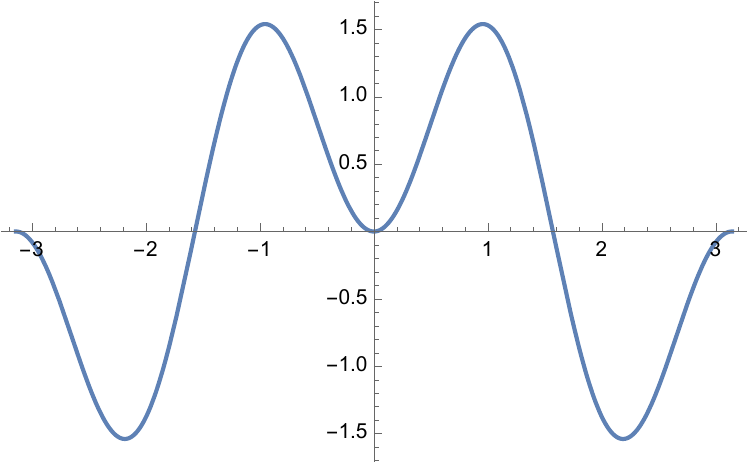}
\caption{See the proof of \Cref{lem:max}. \textbf{Left:} The solid line is the plot of $\cos(t)$ for $t\in[-\tfrac{\pi}{3},\tfrac{\pi}{3}].$ Dashed lines correspond to plots of $\cos(3t-\theta)$ for different choices of $\theta$. \textbf{Right:} Plot of $\sqrt{2} \, P_{\zeta,\pm\pi}(t)=\cos(t)-\cos(3t)$, $t\in[-\pi,\pi]$, for $\zeta=\tfrac{\pi}{4}$.}
\label{fig:cos}
\end{figure}
\begin{proof} We first establish the following.
\begin{claim}\label{claim:maxima} When $\zeta\in[0,\tfrac{\pi}{2})$, all global maxima points of $P_{\zeta,\theta}$ are attained inside of the interval $I_0 := [-\tfrac{\pi}{3},\tfrac{\pi}{3}].$ 
\end{claim}
To prove this claim, note that $[-\pi,\pi) = I_-\cup I_0\cup I_+$ where $I_+ = (\tfrac{\pi}{3},\pi)$ and $I_- = [-\pi,-\tfrac{\pi}{3})$. Let $g(t) := \cos(\zeta)\cos(t)$ and $h(t):=\sin(\zeta)\cos(3t - \theta)$. Note that:
\begin{align}
g(s)&> g(s-\tfrac{2\pi}{3}) &\text{for all $s\in [-\tfrac{\pi}{3},\tfrac{\pi}{3})$}\label{eq:*}\\
g(s)&> g(s+\tfrac{2\pi}{3}) & \text{for all $s\in (-\tfrac{\pi}{3},\tfrac{\pi}{3}]$}\label{eq:**}.
\end{align}

See \Cref{fig:cos}. The claim follows from the following two items:
\begin{itemize}
\item If $t\in I_+$, then $s:= t- \tfrac{2\pi}{3} \in (-\tfrac{\pi}{3},\tfrac{\pi}{3}) \subset I_0$. Thus
 $$P_{\zeta,\theta}(t) = g(t)+h(t) = g(t) + h(t-\tfrac{2\pi}{3}) = g(s+\tfrac{2\pi}{3}) + h(s) < P_{\zeta,\theta}(s) =  P_{\zeta,\theta}(t-\tfrac{2\pi}{3}),$$
where the last inequality is due to \Cref{eq:**}.
\item If $t\in I_-$, then $s:= t+ \tfrac{2\pi}{3} \in [-\tfrac{\pi}{3},\tfrac{\pi}{3})\subset I_0$. Thus
$$P_{\zeta,\theta}(t) = g(t)+h(t) = g(t) + h(t+\tfrac{2\pi}{3}) = g(s-\tfrac{2\pi}{3}) + h(s) < P_{\zeta,\theta}(s) =  P_{\zeta,\theta}(t+\tfrac{2\pi}{3}),$$
where the last inequality is due to \Cref{eq:*}.

\end{itemize}
Assuming $\zeta \in [0,\tfrac{\pi}{2})$, given \Cref{claim:maxima} above, in order for $P_{\zeta,\theta}$ to have more than one global maximum inside of $I_0$, the function $h$ must have at least two local maxima or two local minima in $I_0$. But this requires $\theta = \pm \pi.$  Under this assumption, $P_{\zeta,\pm\pi}(t) = \cos(\zeta)\cos(t) - \sin(\zeta)\cos(3t)$ so that $P_{\zeta,\pm \pi}'(t) = -\cos(\zeta)\sin(t)+3\cos(\zeta)\sin(3t)$. See \Cref{fig:cos}.  Since $ \sin(3t) = \sin(t)(3-4\sin^2(t))$, the critical points inside $I_0$ are therefore $t=0$ together with $t_{\pm}(\zeta)$, the solutions of the equation 
\begin{equation}\label{eq:pm}4\sin^2(t) = 3 - \frac{\cot(\zeta)}{3}.\end{equation} 
Since $P''_{\zeta,\pi}(0) = -\cos(\zeta) + 9\sin(\zeta)$, $t=0$ will correspond to a local minimum whenever $\cot(\zeta) <9$ and it will correspond to a global maximum when $\cot(\zeta)\geq 9$.  For \Cref{eq:pm}   to have two different solutions $t_\pm(\zeta)$ inside the interval  $I_0$ it is necessary and sufficient that $\cot(\zeta) <9$. Whenever this condition holds, $t_\pm(\zeta)$ will be two (different) global maxima points.  Finally, when $\zeta=\zeta_0 = \arccot(9)$, $t=0$ will be the sole global maximum in $I_0$ and $P_{\zeta_0,\pm\pi}'$ will have a triple root at $t=0$. 

\medskip
\noindent
It remains to analyze the case $\zeta = \tfrac{\pi}{2}.$ In this case, $P_{\tfrac{\pi}{2},\theta}(t) = \cos(3t-\theta)$, which certainly has  exactly three different global maxima points.
\end{proof}

\begin{proof}[Proof of \Cref{thm:fiber-3}]
Any point $q(\theta_1,\theta_2,\zeta)$ lying on the boundary of $F_3(0)$ will have more than one closest point in $\gamma_3$.\footnote{That, is there will be at least one point $\gamma_3$ different from $\gamma_0$ that is closest to $q(\theta_1,\theta_2,\zeta)$.} This means that the function $P_3:[-\pi,\pi)\to\R$ defined as
$$P_3(t):=q(\theta_1,\theta_2,\zeta)\cdot \gamma_3(t)=\cos(t-\theta_1)\frac{\cos(\zeta)}{\sqrt{2}} + \cos(3t-\theta_2)\frac{\sin(\zeta)}{\sqrt{2}}$$ will have at least one global maximum, in addition to $t=0$.  

With the notation of  \Cref{lem:max}, $$P_3(t) = \frac{1}{\sqrt{2}}\,P_{\zeta,\theta}(t-\theta_1)\,\,\text{for}\,\,\theta := \theta_2-3\theta_1.$$ Taking into account the considerations above, by \Cref{lem:max}, we must have that $\zeta>\zeta_0$,  $\theta_2 = 3\theta_1\pm \pi$ and that, for any such global maximum,
$t-\theta_1\in[-\tfrac{\pi}{3},\tfrac{\pi}{3}].$ This has to be the case for $t=0$ which gives the following relationship between $\theta_1$ and $\theta_2$:
\begin{equation}\label{eq:3-to-1}\text{ $\theta_2= 3\theta_1 \pm \pi$ for $\theta_1\in[-\tfrac{\pi}{3},\tfrac{\pi}{3}]$.}\end{equation}

On the other hand, note that by \Cref{prop:convex} we must have  
$$F_3(0)\subset \{v\in \R^4|\,v\cdot \dot{\gamma}_3(0)=0\}=\Sigma_0.$$
Since $\dot{\gamma}_3(0)=\frac{1}{\sqrt{2}}\big(0,1,0,3\big),$ every point $(x,y,z,w)$ in the hyperplane $\Sigma_0$ satisfies $y+3w=0$. For a point $q(\theta_1,\theta_2,\zeta)$ in $\Sp^3\cap \Sigma_0$, expressed in Hopf coordinates, this gives the condition 
\begin{equation}\label{eq:plane-hopf-coords}\sin(\theta_1)\cos(\zeta)+3\sin(\theta_2)\sin(\zeta)=0.\end{equation}
By \Cref{eq:3-to-1}, if $q(\theta_1,\theta_2,\zeta)$ is in $\partial F_3(0)$, we have that $\theta_2=3\theta_1\pm\pi$ so that   \Cref{eq:plane-hopf-coords}  becomes
$$\sin(\theta_1)\cos(\zeta) = 3\sin(3\theta_1)\sin(\zeta).$$
Via the formula $\sin(3\alpha) = \sin(\alpha)\big(3-4\sin^2(\alpha)\big)$ we obtain from the above condition that 
$$\cos(\zeta) = 3\big(3-4\sin^2(\theta_1)\big)\sin(\zeta)$$ from which the first claim follows. The second claim follows from \Cref{prop:convex}. 
\end{proof}

\subsection{The modulus of discontinuity of $\psi_3$ is minimal}\label{sec:mod-disc-3-min}
By combining  \Cref{coro:disc-voro} with \Cref{thm:fiber-3} and \Cref{thm:disc-lower-bound} we have that the modulus of discontinuity of $\psi_3$ is minimal.
\begin{theorem}\label{thm:disc-3}
    $\disc(\psi_3)=\tfrac{2\pi}{3}.$
\end{theorem}
Indeed, this theorem is  proved by directly exploiting the precise description of  $\partial F_3(0)$ given by \Cref{thm:fiber-3}. Since, by \Cref{prop:moddis}, distortion is always lower bounded by modulus of discontinuity, \Cref{thm:disc-3} can be seen as a preamble to  and a consequence of  \Cref{thm:dis-3} below. We however include a standalone proof here because this proof contains interesting ideas related to the effect of the rotation $T_t$ on $\partial F_3(0)$.  This theorem will be generalized, via completely different arguments, in \Cref{thm:disc-gen-k}. The arguments therein  exploit a connection between the Voronoi cells induced by $\gamma_{2k+1}$ and known results about the facial structure of the Barvinok-Novik polytope. 
\begin{proof}[Proof of \Cref{thm:disc-3}]
By \Cref{thm:disc-lower-bound}, it suffices to prove that $\disc(\psi_3)\leq \tfrac{2\pi}{3}$. 

Assume that $q_0\in \partial F(0) \cap \partial F(t)$ for some $t\in[0,\pi]$. Then, we have: 
\begin{itemize}
\item Since, by \Cref{rem:rot}, $\partial F(t)=T_t(\partial F(0))$, there exists $q_0'\in \partial F(0)$ such that $q_0 = T_t(q_0').$
\item By \Cref{rem:boundary-3}, there exist $\theta_0,\theta_0'\in [-\tfrac{\pi}{3},\tfrac{\pi}{3})$ such that $$\mbox{$q_0 = q(\theta_0,3\theta_0+\pi,\zeta(\theta_0)$ and $q_0' = q(\theta_0',3\theta_0'+\pi,\zeta(\theta_0'))$}.$$
\item By \Cref{coro:preserv-tori} we have $$q_0 =  T_t(q'_0) = T_t (q(\theta_0',3\theta_0'+\pi,\zeta(\theta_0'))) = q(\theta_0'+t,3\theta_0'+\pi+3t,\zeta(\theta_0')).$$ Hence, through the equality
$$q(\theta_0,3\theta_0+\pi,\zeta(\theta_0)=q_0 = q(\theta_0'+t,3\theta_0'+\pi+3t,\zeta(\theta_0'))$$
we see that the following two conditions must hold: $\theta_0 = \theta_0'+t$ and $\zeta(\theta_0) = \zeta(\theta_0').$ Via the explicit expression for $\zeta(\cdot)$ given in \Cref{eq:def-eta-theta}, we see that the second condition implies that it must then hold that $\theta_0 = \pm \theta_{0}'$. The case $\theta_0=\theta_0'$ leads to $t=0$ via the first condition. The  case $\theta_0 = -\theta_0'$ gives $\theta_0' = -\tfrac{t}{2}$.
\item Finally, the condition that $\theta_0' \in [-\tfrac{\pi}{3},\tfrac{\pi}{3})$ yields that it must be that $t\leq \tfrac{2\pi}{3}$. In other words, $q_0 = \bar{q}(\tfrac{t}{2})$ and $q_0' = \bar{q}(-\tfrac{t}{2})$; see \Cref{fig:rot-fibers}.
\end{itemize}
\end{proof}

\begin{figure}
    \centering
    \includegraphics[width = 0.75\linewidth]{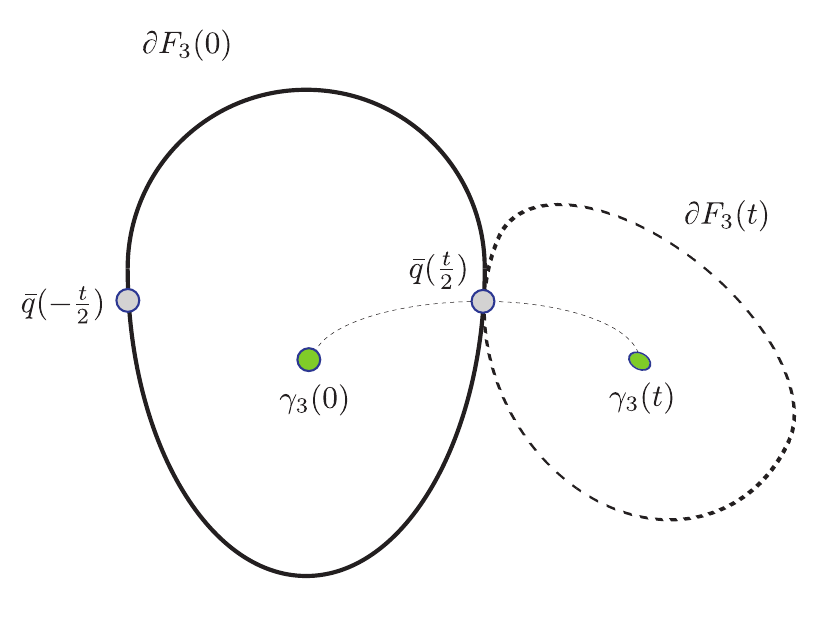}
    \caption{The point $\bar{q}(\tfrac{t}{2})$ is in the intersection $\partial F_3(0)\cap \partial F_3(t)$; see the proof of \Cref{thm:disc-3}.} 
    \label{fig:rot-fibers}
\end{figure}

\subsection{The distortion of $R_3$ is minimal \extension}\label{sec:dis-3-min}
Exploiting the characterization of $F_3(0)$ established in \Cref{thm:fiber-3} we now have the following result.

\begin{theorem}\label{thm:dis-3}
$\dis(R_3)\leq \frac{2\pi}{3}.$
\end{theorem}

Via ideas described in \Cref{sec:simple-conditions} and \Cref{thm:fiber-3}, the proof is reduced to checking two inequalities involving elementary trigonometric functions which can we verify with the aid of a computer program.\footnote{We use the term ``computer assisted proof" for lack of a better one.}

\begin{proof}[Computer assisted proof \faLaptopCode]
Given the above, it is enough to verify condition $B^*(\frac{2\pi}{3})$ (from \Cref{eq:Baster}), that is:
\begin{align*}
 & &   d_1(0,t) - \tfrac{2\pi}{3} \leq d_{3}(q,T_t q')  && \forall \, q,q'\in \partial  F_{3}(0),\, |t| \in [\tfrac{2\pi}{3}, \pi].
\end{align*}
For $t\in[\tfrac{2\pi}{3},\pi]$,  this is equivalent to the condition 
\begin{align*}
 &(*) &   \cos(t - \tfrac{2\pi}{3}) \geq q \cdot T_t q'  && \forall \, q,q'\in \partial  F_{3}(0),\, |t| \in [\tfrac{2\pi}{3}, \pi].
\end{align*}

Since the RHS involves points on $q,q'\in\partial F_3(0)$, via \Cref{rem:boundary-3}, writing $q=q(\theta)$ and $q'=q(\theta')$, we see that 
$$q\cdot T_t q' = F_t(\theta,\theta'):=\cos(\theta-\theta' + t) \cos(\zeta)\cos(\zeta') + \cos(3(\theta-\theta' + t)) \sin(\zeta)\sin(\zeta')$$
where we've written $\zeta := \zeta(\theta)$ and $\zeta':=\zeta(\theta')$ for conciseness. Condition $(*)$ is therefore equivalent to the following condition involving elementary functions
\begin{align*}
 &(**) &   \cos(t - \tfrac{2\pi}{3}) \geq F_t(\theta,\theta')  && \forall \, \theta,\theta'\in [-\tfrac{\pi}{3},\tfrac{\pi}{3}],\, t \in [\tfrac{2\pi}{3}, \pi].
\end{align*}
Analogously, the case $t\in[-\pi,-\tfrac{2\pi}{3}]$ leads to the condition 
\begin{align*}
 &(***) &   \cos(t + \tfrac{2\pi}{3}) \geq F_t(\theta,\theta')  && \forall \, \theta,\theta'\in [-\tfrac{\pi}{3},\tfrac{\pi}{3}],\, t \in [-\pi,-\tfrac{2\pi}{3}].
\end{align*}
These two conditions involving elementary functions were  verified with the assistance of Matlab \cite[Function \texttt{TestIneqR3.m}]{dgh-github}.
\end{proof}

\section{The case $\Sp^1$ versus $\Sp^{2k+1}$}\label{sec:gen-k}

In this section we describe some results for the case of $\Sp^1$ versus $\Sp^{2k+1}$ which are applicable to the case $k\geq 2$.

\subsection{Connections to the Barvinok-Novik polytope}

The convex hull of the image of the TMC is nowadays  known  as the \emph{Barvinok-Novik polytope}. It is defined as
$$\mathcal{B}_{2k} := \conv(\gamma_{2k-1}(\Sp^1)).$$

In \Cref{sec:cov} below we will use results about the facial structure of this polytope to establish some results pertaining to the TMC-EPCs. 

Smilansky \cite{smilansky1990bi} studied $\mathcal{B}_4$ and Barvinok and Novik \cite{barvinok2008centrally}, and then Vinzant  \cite{vinzant2011edges} and  Barvinok-Lee-Novik \cite{barvinok2013neighborliness} studied the general case; see also \cite{barvinokbeauty}.  

Smilansky obtained the following\footnote{Smilansky's results are more general than what we state here.} characterization of the facial structure of $\mathcal{B}_4$.
\begin{theorem}[{\cite[Theorem 1]{smilansky1990bi} and \cite[Theorem 4.1]{barvinok2008centrally}}]\label{thm:smilansky} The proper faces of $\mathcal{B}_4$ are
\begin{itemize}\item[(0)] the $0$-dimensional faces (vertices) are $\gamma_{3}(t)$ for $t\in \Sp^3$.
\item[(1)] the $1$-dimensional faces (edges) are the segments
$$[\gamma_3(t_1),\gamma_3(t_2)]$$ where $t_1\neq t_2$ and $d_1(t_1,t_2)\leq \tfrac{2\pi}{3}$; and
\item[(2)] the $2$-dimensional faces of $\mathcal{B}_4$ are all the equilateral triangles 
$$\Delta_t:=\conv\big(\gamma_{3}(t),\gamma_3(t-\tfrac{2\pi}{3}),\gamma_3(t+\tfrac{2\pi}{3}))\big),\,\,\text{$t\in\Sp^1$}.$$
\end{itemize}
\end{theorem}

Whereas a complete characterization of the facial structure of $\mathcal{B}_{2k}$ for general $k$ does not seem the be yet available \cite[Section 2]{barvinokbeauty},\footnote{Barvinok notes ``For larger [$k$], we have only some fragmentary information regarding the facial
structure of the convex hull of ...[$\gamma_{2k+1}$].."} the following result provides a complete description of  its edges.
\begin{theorem}[{\cite[Theorem 1.1]{barvinok2008centrally} and \cite[Theorem 1]{vinzant2011edges}}]\label{thm:gen-edges} For $t_1\neq t_2$ in $\Sp^{1}$, the segment $[\gamma_{2k-1}(t_1),\gamma_{2k-1}(t_2)]$ is an exposed edge of $\mathcal{B}_{2k}$ if and only if $d_1(t_1,t_2)\leq \delta_{k-1}$.
\end{theorem}

The following result provides partial information on higher dimensional faces. 
\begin{theorem}[{\cite[Theorem 1.2]{barvinok2008centrally} and \cite[Theorem 1.1]{barvinok2013neighborliness}}] \label{thm:gen-other} For every $k$ there exists a number $\pi>\phi_k>\tfrac{\pi}{2}$ such that if $\ell\leq k$ and $A=\{t_1,\ldots,t_\ell\} \subset \Sp^1$ are $\ell$ distinct points contained in an arc with length at most  $\phi_k$, then 
$\conv\big(\gamma_{2k-1}(A)\big)$ is an $(\ell-1)$-dimensional exposed face of $\mathcal{B}_{2k}.$
\end{theorem}

\begin{remark}
Certainly, through \Cref{thm:smilansky}, $\phi_2 = \tfrac{2\pi}{3}$. The authors of \cite{barvinok2013neighborliness}  verified that:
\begin{itemize}
\item $\lim_k\phi_k = \tfrac{\pi}{2}$,
\item $\phi_3 = \pi - \arccos\left(\tfrac{3-\sqrt{5}}{2}\right)$,
\end{itemize}
and also determined the precise value of $\phi_4$. 
\end{remark}

\begin{remark}[A family of simplicial faces]\label{rem:gen-faces}
In \cite[page 86]{barvinok2008centrally} Barvinok and Novik also describe the following 1-parameter family of $(2k-1)$-dimensional simplicial faces of $\mathcal{B}_{2k}$:
$$\Delta_t:=\conv\left(T_t\big(\gamma_{2k-1}(Q_{2k-1})\big)\right),\,\,\mbox{where}\,\, Q_{2k-1}:=\big\{0,\tfrac{2\pi}{2k-1},\ldots,\tfrac{2\pi (2k-2)}{2k-1}\big\}$$ are the vertices of a regular odd polygon inscribed in $\Sp^1$. The authors observe that, due to properties of the TMC, $\Delta_t$ is a $(2k-2)$-dimensional \emph{regular} simplex. 
\end{remark}

In \cite[Proposition 5.2]{sinn2013algebraic} the author establishes that, in fact,  all faces of $\mathcal{B}_{2k}$ are simplicial. Combining this with \Cref{thm:gen-edges} we obtain the following.
\begin{corollary}\label{coro:sinn}
    Every face $\sigma$ of $\mathcal{B}_{2k}$ is a simplex of the form $\sigma = \conv(\gamma_{2k-1}(A))$ for some finite subset $A\subset \Sp^1$ such that $\diam(A)\leq \delta_{k-1}$.
\end{corollary}

\subsection{An application to \Cref{conj:cov} \extension}\label{sec:cov}
Note that the boundary $\partial \mathcal{B}_{2k+2}$ is homeomorphic to $\Sp^{2k+1}$ and its  projection onto  $\Sp^{2k+1}$ gives a simplicial decomposition $\mathfrak{D}_{2k+1}$ of $\Sp^{2k+1}$,
$$\Sp^{2k+1} = \bigcup_{\sigma\in \mathfrak{D}_{2k+1}}\sigma,$$ into geodesic (and therefore convex) simplices of different dimensions. This then suggests that, in order to tackle \Cref{conj:cov}, one could exploit the precise description of the cells given by \Cref{thm:smilansky}, \Cref{thm:gen-edges}, \Cref{thm:gen-other} and \Cref{rem:gen-faces} in order to 
verify  $$\min_{t\in \Sp^1}d_{2k+1}(q,\gamma_{2k+1}(t))\leq \frac{\delta_{k}}{2}\,\,\mbox{for all $q\in \sigma \in \mathfrak{D}_{2k+1}$}.$$

Using this strategy we
are able to establish  \Cref{conj:cov} for the case $k=1$. We also provide partial information about other cases which could be useful as more information becomes available regarding the faces of the Barvikok-Novik polytope. Some of the arguments in the proof of \Cref{thm:cov-3} can be generalized beyond the case $k=1$.

\begin{proposition}\label{thm:cov-3}
For all $q\in \Sp^{3}$ one has 
$$\min_{t\in \Sp^1}d_{3}(q,\gamma_{3}(t))\leq \frac{\pi}{3}.$$
\end{proposition}

\begin{lemma}[\extension]\label{lemma:cov-1-dim}
For all $q\in\sigma\in \mathfrak{D}_{2k+1}$, where $\sigma$ has dimension $1$, one has
$$\min_{t\in \Sp^1}d_{2k+1}(q,\gamma_{2k+1}(t))\leq \tfrac{\delta_k}{2}.$$
\end{lemma}
\begin{proof}
According to \Cref{thm:gen-edges}, any such $\sigma$ is a geodesic segment  joining $\gamma_{2k+1}(t)$ and $\gamma_{2k+1}(s)$, where $d_1(t,s)\leq \tfrac{2\pi}{3}$. Any $q$ on that segment
will be at distance at most $$\mu_{k}(t,s):=\tfrac{1}{2}d_{2k+1}(\gamma_{2k+1}(t),\gamma_{2k+1}(s))$$ from the set consisting of the two endpoints of the geodesic segment.  Then, via item 4 of \Cref{prop:props-hk}, we have
$$\min_{\tau\in I}h_k(\tau) \geq \cos(\delta_k)$$
which implies that, for $t,s\in\Sp^1$ such that $d_1(t,s)\leq \delta_k$,  
$$\mu_k(t,s)\leq \frac{1}{2}\arccos\left(\min_{\tau\in I}h_k(\tau)\right)\leq \frac{1}{2}\delta_k.$$
\end{proof}

\begin{proof}[Proof of \Cref{thm:cov-3}] 
According to \Cref{coro:sinn}, $\Sp^{3}$ can be decomposed into simplicial cells $\sigma\in \mathfrak{D}_{3}$ of dimension at most $2$ such that $\sigma = \conv(\gamma_{3}(A))$, where  $A$ is a subset of $\Sp^1$ with diameter at most $\delta_1 = \tfrac{2\pi}{3}$. Note that  the cardinality of $S$  can be at most $3$. We now consider the following cases:

\begin{itemize}
\item dimension zero cells corresponding to points lying on $\gamma_{3}$;
\item dimension one cells: geodesic segments joining points $\gamma_{3}(t)$ and $\gamma_{3}(s)$ s.t. $d_1(t,s)\leq \tfrac{2\pi}{3};$
\item  dimension $2$ cells: regular spherical simplices arising as the projection on $\Sp^{3}$ of the equilateral triangles determined by triples of points of the form $\gamma_3(t)$, $\gamma_3(t+\tfrac{2\pi}{3})$ and $\gamma_3(t-\tfrac{2\pi}{3})$.
\end{itemize}

The second case can be dealt  via \Cref{lemma:cov-1-dim} (for $k=1$). The third case leads to considering, for each $t\in \Sp^1$, the point inside the aforementioned  spherical equilateral triangle that is as far as possible from its vertices. This point will be the center $(0,0,\cos(3t),\sin(3t))$ of the triangle  and this point is at distance $\arccos\left(\tfrac{1}{\sqrt{2}}\right)=\tfrac{\pi}{4}<\tfrac{\pi}{3}$ from the vertices. \end{proof}

In \Cref{sec:gen-k} we provide a precise connection between the facial structure   $\mathcal{B}_{2k+2}$ ($\mathcal{D}_{2k+1}$) and the Voronoi tiling of $\Sp^{2k+1}$ induced by the TMC.

\subsection{Connecting  structure of $\mathcal{B}_{2k}$ to  Voronoi tiling induced by $\gamma_{2k+1}$}

The following proposition establishes a duality between the facial structure of $\mathcal{B}_{2(k+1)}$ and the Voronoi tiling of $\Sp^{2k+1}$ induced by $\gamma_{2k+1}$. 
\begin{proposition}
\label{prop:voro-conv}
Let  $t_1,t_2,\ldots,t_\ell\in \Sp^1$ be distinct points. Then, $$ \bigcap_{i=1}^\ell \partial F_{2k+1}(t_i) \neq \emptyset \Longleftrightarrow\conv\big(\{\gamma_{2k+1}(t_1),\ldots,\gamma_{2k+1}(t_\ell)\}\big)\text{\,\,is an exposed face of\,\,} \mathcal{B}_{2(k+1)}.$$
\end{proposition}

\begin{figure}[h]
    \centering
    \includegraphics[width = \linewidth]{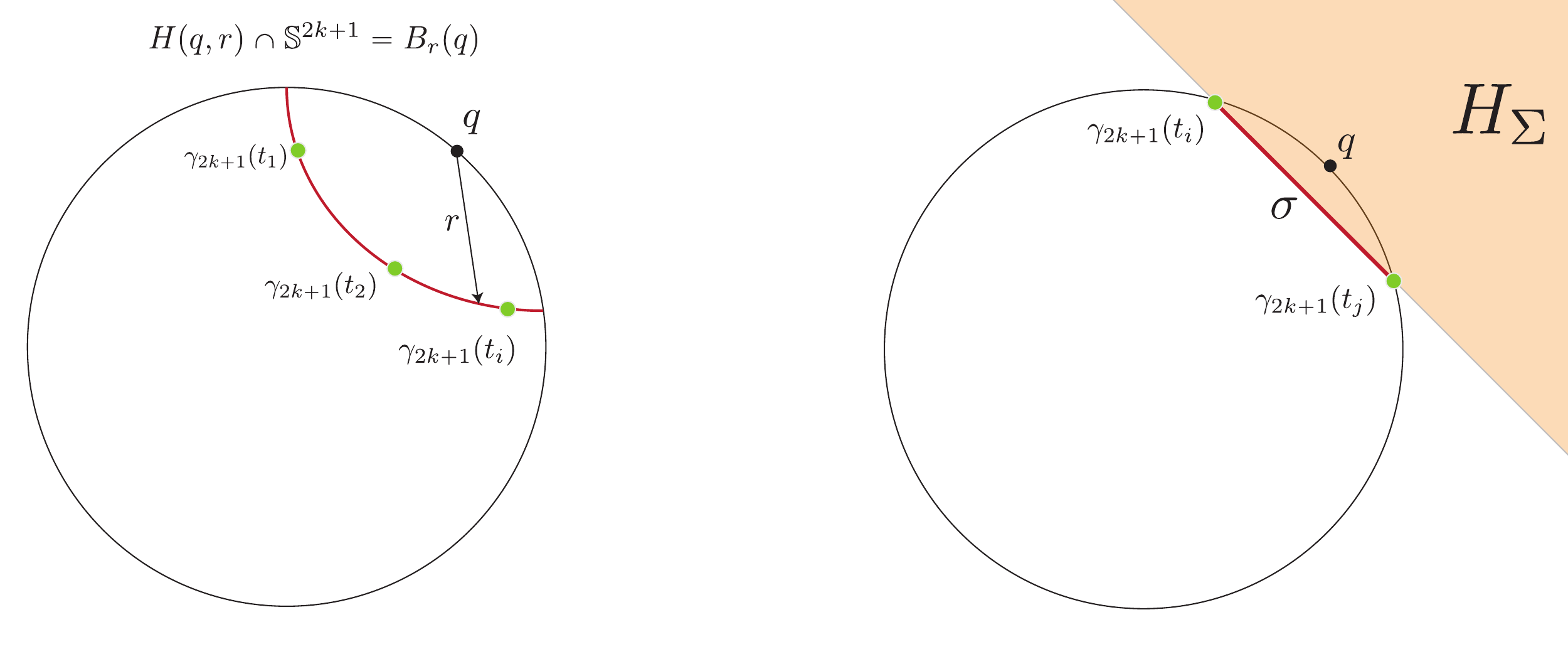}
    \caption{See the proof of \Cref{prop:voro-conv}. \textbf{Left:} the point $q\in \partial F_{2k+1}(t_i)$ is equidistant to all $\gamma_{2k+1}(t_i)$. The red line represents the boundary of  the spherical cap (geodesic ball) $B_r(q) = H(q,r)\cap \Sp^{2k+1}$.  \textbf{Right:} $\sigma$ is an exposed face of $\mathcal{B}_{2(k+1)}$. The intersection  $H_\Sigma \cap \Sp^{2k+1}$ is a spherical cap.}
    \label{fig:duality}
\end{figure}

\begin{remark}[\extension]
   Note that \Cref{lem:max} (or alternatively, \Cref{thm:fiber-3}) together with \Cref{prop:voro-conv} permits recovering \Cref{thm:smilansky} which provides a complete characterization of $\mathcal{B}_4$. The duality between the intersection pattern of Voronoi cells and the facial structure of the Barvinok-Novik polytope seems to be an interesting direction to further explore.  In particular, the current partial knowledge about the facial structure of $\mathcal{B}_{2k+2}$ can be utilized to shed light on the structure of $\partial F_{2k+1}(0)$. On the other hand, the fact that currently the full characterization of the facial structure of $\mathcal{B}_6$  is not known suggests that determining the precise shape of $\partial F_5(0)$
 might pose some challenges. One expects these challenges to  be present in terms of determining the shape of $\partial F_{2k+1}(0)$ for all $k\geq 2$.\end{remark}

\begin{remark}
    The connection between the convex hull $\conv(A)$ of a \emph{finite} subset $A$ of $\Sp^n$ and the Voronoi tiling of $\Sp^n$ it induces already appears in the PhD thesis of Brown \cite{brown1979geometric}; see also \cite{brown1979voronoi,fortune2017voronoi}. We include a proof of \Cref{prop:voro-conv} for completeness since this equivalence might not be well known.
\end{remark}

The following corollary to \Cref{thm:gen-edges} and \Cref{prop:voro-conv} will be immediately useful.
\begin{corollary}\label{coro:voro-intersect}
For all $t,s\in\Sp^1$, 
$$\partial F_{2k+1}(s)\cap F_{2k+1}(s)\neq \emptyset \Longleftrightarrow d_1(s,t)\leq \delta_k.$$ 
\end{corollary}

\begin{proof}[Proof of \Cref{prop:voro-conv}]

Suppose that $q\in \partial F_{2k+1}(t_i)$ for $i=1,\ldots,\ell$. This implies that $q$ is equidistant to all $\gamma_{2k+1}(t_i)$. Let $r:=d_{2k+1}(q,\gamma_{2k+1}(t_i))$ be the value of the common distance. Then, the open geodesic ball $B_r(q)$ does not contain any point $\gamma_{2k+1}(t)$, $t\in \Sp^1$, in its interior. Since $B_r(q)$ is the intersection of $\Sp^{2k+1}$ with the following half-space (see  the left panel of \Cref{fig:duality})
$$H(q,r):=\{x\in \R^{2k+2}|\big(x-q\cos(r)\big)\cdot q \geq 0\},$$
we have that $\conv\big(\{\gamma_{2k+1}(t_1),\ldots,\gamma_{2k+1}(t_\ell)\}\big)$ is an exposed face of $\mathcal{B}_{2(k+1)}$.

For the converse, assume that $\sigma = \conv\big(\{\gamma_{2k+1}(t_1),\ldots,\gamma_{2k+1}(t_\ell)\}\big)$ is an exposed face of $\mathcal{B}_{2(k+1)}$. Let $\Sigma$ be a supporting hyperplane for this face, let $H_\Sigma$ be the associated half-space such that  $\mathcal{B}_{2(k+1)}\cap H_\Sigma = \sigma$ and let $q$ be the center of the spherical cap $\Sp^{2k+1}\cap H_\Sigma$; see  the right panel of \Cref{fig:duality}. Let $r$ be the (spherical) radius of this cap, i.e. such that $\Sp^{2k+1}\cap H_\Sigma = \overline{B_r(q)}$.  Then, by construction, $q$ is at distance $r$ to $\gamma_{2k+1}(t_i)$, for $i=1,\ldots \ell$, and its distance to any other $\gamma_{2k+1}(t)$, $t\in\Sp^1$, is at least $r$. Hence, $q$ lies in the intersection $\cap_{i=1}^\ell \partial F_{2k+1}(t_i)$.
\end{proof}

\subsection{The modulus of discontinuity of $\psi_{2k+1}$ is minimal}

The theorem below is obtained as consequence of current knowledge of the structure of the Barvinok-Novik $\mathcal{B}_{2k}$ and the relationship between convex hulls and Voronoi tilings induced by points on spheres. This theorem can be interpreted as emphasizing one particular aspect in which the TMC is optimal.  

\begin{theorem}\label{thm:disc-gen-k}
$\disc(\psi_{2k+1})=\disc(\psi_{2k})=\delta_{k} = \frac{2k\pi}{2k+1}.$
\end{theorem}

\begin{proof}[Proof of \Cref{thm:disc-gen-k}]
By \Cref{thm:disc-lower-bound} and \Cref{prop:dis2k}, it suffices to prove that $\disc(\psi_{2k+1})\leq \delta_k$. To this end, we are going to invoke \Cref{coro:disc-voro} which states that
$$\disc(\psi_{2k+1}) = \min\{t\in [0,\pi]|\,\partial F_{2k+1}(0)\cap \partial F_{2k+1}(t)\neq \emptyset\}.$$
Assume that $t\in[0,\pi]$ is such that $\partial F_{2k+1}(0)\cap \partial F_{2k+1}(t)\neq \emptyset.$ By \Cref{coro:voro-intersect}, $d_1(0,t)\leq \delta_k$. Thus, $\disc(\psi_{2k+1})\leq \delta_k$.
\end{proof}

\subsection{Additional results stemming from \Cref{prop:voro-conv} \extension}

\section{Results for the case $\Sp^m$ versus $\Sp^n$ \faClock}

\section{Historical account and connections \faClock} \label{sec:hist}

Our project about the precise determination of the Gromov-Hausdorff distance between spheres was started around 2003 in the context of the PhD thesis of the first author (FM). Part of the thesis was devoted to applications of the GH distance in shape matching/comparison applications. This required developing some code for computing/estimating the distance. However, the code that was developed was not provably correct, in the sense that it used a number of heuristics in order to estimate the value of the distance. One observation made back then was that if the actual values of GH distance on a collection of canonical shapes were determined via theoretical methods then one would be able to assess the quality of said software by comparing the theoretically predicted value with the value produced by the software. It was eventually established that the computational problem posed by the GH distance is NP hard in the PhD work of Schmiedl \cite{schmiedl2017computational}.

In 2007, during an Applied Topology seminar at Stanford, Tigran Ishkhanov suggested to FM the possibility of invoking the Borsuk-Ulam theorem in the context of the problem of computing the GH distance between spheres. This did not seem immediately useful as the standard version of the BU theorem is only applicable to continuous functions which are not, a priori, present in the definition of the GH distance. During the same seminar meeting, Gunnar Carlsson suggested using the stability of persistent homology for obtaining efficient lower bounds for the GH distance between spheres.  This thread was explored in the research leading to \cite{lim2020vietoris} and \cite{memoli2019persistent}.  The former approach would be revisited later (see below).

Around 2013/2014 this problem was energetically discussed in group meetings at OSU in which the second author (ZS) participated.  Circa 2014/2015 ZS first constructed and experimentally tested the correspondences $R_{\gamma_n}$ between $\Sp^1$ and $\Sp^n$ for $n=2,3,4,5$.  ZS found the shape of such curves through painstaking trial and error.\footnote{Aided by classical multidimensional scaling methods in order to visualize the progressively better correspondences he obtained.} 

Soon after that, FM  tested these correspondences and constructed and tested the correspondence from \Cref{sec:variant-R2} (between $\Sp^1$ and $\Sp^2$)  as way of obtaining a correspondence which was somehow simpler than $R_{\gamma_2}.$ This was then generalized to those correspondences described in \Cref{sec:corr-sn-snp1} for spheres $\Sp^m$ and $\Sp^n$ of different dimension.

\medskip
In the next few years the following took place:
\begin{itemize}
    \item FM and ZS continued to experimentally test  the correspondences $R_{\gamma_n}$ while they also
    \item attempted to theoretically determine the distortion of these correspondences. 

\item Versions of these EPC correspondences applicable to spheres $\Sp^m$ and $\Sp^n$ of different dimension were also developed and tested extensively.
\item Around 2015 Sunhyuk Lim (SL) became interested in the project and came on board to help develop the ideas that would eventually become \cite{lim2021gromov} and also those that led to the related project \cite{lim2020vietoris}.
    
    \item At some point in 2014/2015 they found a cartoonization of $R_{\gamma_2}$ for which we were able to precisely determine its distortion. This cartoonization is explained in detail in \cite[Appendix D]{lim2021gromov-arxiv}. 

    \item Neither ZS or FM were aware of the connection between the correspondences $R_{\gamma_n}$ and the TMC or the Barvinok-Novik polytope.  In early 2017 Henry Adams visited OSU and, in the course of a conversation, ZS and FM  shared with him that they were working on the problem of determining the GH distance between spheres and described the TMC-EPC. During that conversation, Henry shared with FM and ZS that he had been learning about the Barvinok-Novik polytope and that he recognized that the curve they had been contemplating was known as the (symmetric) TMC in the literature about polytopes. Henry encouraged FM and ZS to look into that connection. Henry's study of the TMC and the BN polytope eventually led to his joint work with Johnathan Bush and Florian Frick \cite{adams2020metric,adams2023topology}. In a rather precise sense, ideas from both  \cite{adams2020metric} and \cite{lim2021gromov} were eventually combined in order to obtain the results in \cite{adams2022gromov}. 

\item One of the difficulties encountered when trying to determine the precise value of the distortion of $R_{\gamma_n}$ and in establishing that it is optimal was the lack of knowledge about tight lower bounds for GH between spheres. The idea that was explored initially was to use the GH stability of persistent homology of Vietoris-Rips complexes to help in this regard. This led to \cite{lim2020vietoris} and \cite{memoli2019persistent} but did not end up giving tight lower bounds; see the discussion in \cite[Section 9.3.2]{lim2020vietoris} and \cite[Remark 1.13]{lim2021gromov}. 

\item  At some point in 2019/2020 the authors of \cite{lim2021gromov} explored the idea of establishing suitable versions of the Borsuk-Ulam theorem to quantitavely obstruct the existence of low distortion correspondences between spheres of different dimension. This was FM's interpretation of a suggestion made by Tigran Ishkhanov around 2007 (it took several years to eventually come around to this possibility which turned out to be very fruitful). In the summer of 2020, while reading \cite{matouvsek2003using} we found a mention of a version of the Borsuk-Ulam theorem applicable to discontinuous functions due to Dubins and Schwarz \cite{dubins1981equidiscontinuity}. After some massaging (via the so called `helmet trick' \cite[Lemma 5.7]{lim2021gromov}) they were able to invoke this theorem in order to conclude that, for $n>m$, $\dgh(\Sp^n,\Sp^m)\geq \tfrac{1}{2}\zeta_m$ where $\zeta_m:=\arccos\big(\tfrac{-1}{m+1}\big).$ 

\item When $m=1$ and $n=2,3$,they  experimentally determined that this lower bound, $\tfrac{\pi}{3}$, was matched by the distortions of $R_{\gamma_2}$ and $R_{\gamma_3}$ yet they were not able to prove this mathematically. This led to developing the correspondences used in \cite[Proposition 1.16 and 1.18]{lim2021gromov} and in \cite[Appendix D]{lim2021gromov-arxiv} (for $n=2,3$) through the process of ``cartoonization" mentioned earlier, which were substantially easier to analyze. A first version of \cite{lim2021gromov} was completed in 2021 \cite{lim2021gromov-arxiv}.   

\item In 2020/2021 Henry noticed a similarity between the table in \cite[Figure 2]{lim2021gromov} and the table on \cite[page 11]{bush2019proposal} and \cite[page 80]{bush2021topological} which made him suspect the existence of a quantitative connection between \cite{lim2021gromov} and \cite{adams2020metric}. See \cite[Question 8.12]{adams2022gromov}.

\item From 2021 to 2022 a Polymath style group was formed with participants from Colorado State University, Carnegie Mellon, Ohio State and the Freie Universitet in Berlin. Many developments that combined ideas related to different threads came out of that project. See  \cite{adams2022gromov}.

\item One particular result obtained in the course of the Polymath activity was the lower bound given in \Cref{thm:lower-bound}, namely $\dgh(\Sp^1,\Sp^{2n})\ge\frac{\pi n}{2n+1}$ and $\dgh(\Sp^1,\Sp^{2n+1})\ge\frac{\pi n}{2n+1}$, for arbitrary $n\geq 1$. This improved upon the lower bound $\dgh(\Sp^1,\Sp^n)\geq \tfrac{\pi}{3}$ for arbitrary $n$ arising from \cite{lim2021gromov}. Henry conjectured these inequalities were in fact equalities.
By providing a clear, explicit, lower bound, this has a direct impact on the analysis of the correspondences $R_{\gamma_n}$: to prove they are optimal, one just needs to prove their distortion is bounded above by twice the lower bound.

\item Prompted by Henry, starting early in the course of the Polymath, FM described the general EPC idea as well as the TMC-EPC version to all the participants. During one of the subsequent meetings, Johnathan Bush described having experimented with the case of $R_{\gamma_{7}}$ and  obtaining results that matched the lower bound $\tfrac{7\pi}{15}$ predicted by \Cref{thm:lower-bound}. In subsequent meetings, FM described cartoonizations of the TMC-EPC via geodesic segments, generalizations,  as well as some germs of the results presented in this writeup.

 \item In July 2022, in the context of the Polymath, Amzi Jeffs and Michael Harrison started exploring a construction of correspondences between $\Sp^m$ and $\Sp^{n}$ arising through first identifying suitable finite centrally symmetric point sets on each sphere in order to then partition spheres into Voronoi cells. These correspondences are structurally related to  and inspired by  the ones described in \cite{lim2021gromov} and especially to those in \cite[Appendix D]{lim2021gromov-arxiv}. By carefully designing these points sets, Amzi and Michael managed to prove they were optimal for all $(n,m) = (1, 2k)$, $k\geq 1$. This led to the results in \cite{harrison2023quantitative} containing the first complete description of a family of optimal correspondences between $\Sp^1$ and \emph{all} even dimensional spheres.

\item More or less simultaneously, in November 2022, Amzi and Michael started thinking of possible ways of constructing optimal correspondences between $\Sp^1$ and all odd dimensional spheres. For this they explored constructions inspired by the TMC-EPC (especially cartoonizations via piecewise geodesic curves). In April 2023 they reported  having been able to prove that the correspondences they constructed were optimal. An upcoming update to \cite{harrison2023quantitative} will describe this construction and establish its optimality. Therefore, these results in combination with the ones described in the previous bullet point, provide the first complete answer to \Cref{q:central} for the value $m=1$.
     
\end{itemize}

\newcommand{\etalchar}[1]{$^{#1}$}

\end{document}